\documentclass[11pt]{article}
\usepackage{amsmath,amsthm,amsfonts}
\usepackage{fullpage}
\usepackage{graphicx}
\usepackage{caption}
\usepackage{parskip}
\usepackage{bm}
\usepackage{hyperref}
\usepackage{float}

\usepackage[nottoc]{tocbibind}

\newtheorem{theorem}{Theorem}[section] 
\newtheorem{lemma}[theorem]{Lemma}
\newtheorem{cor}[theorem]{Corollary}
\newtheorem{conj}[theorem]{Conjecture}
\theoremstyle{definition}
\newtheorem{defn}[theorem]{Definition}

\numberwithin{equation}{section}

\newcommand{\Z}{\mathbb Z}
\newcommand{\N}{\mathbb N}

\newcommand{\usection}[1]{\section*{#1}
\addcontentsline{toc}{section}{#1}}

\begin{document}

\title{On groups with D-finite cogrowth series}
\author{Mudit Aggarwal, Murray Elder, and Andrew Rechnitzer}
\maketitle

\begin{abstract}
The cogrowth series of a group with respect to a finite generating set is an important combinatorial quantity that seems very difficult to compute exactly, as evidenced by the scarcity of known examples. In this paper, we give a particular infinite family of presentations for which the cogrowth series can be determined as the constant term of an algebraic function, which shows that it is D-finite and, with more work, not algebraic.

Our proof exploits the fact that for a particular choice of subgroup, the corresponding Schreier graph has finite tree width, and by considering paths in the cosets and the Schreier graph separately, we are able to construct a system of generating functions which count paths. We find the asymptotics of this system to conclude that the groups have D-finite but non-algebraic cogrowth series.

We also apply our method to some additional examples which have some similarities with the infinite family above, and again show they have D-finite but non-algebraic cogrowth series.

These examples lend some support to the
conjecture that if a group has an algebraic cogrowth series, then it must be virtually-free, and adds to the small collection of known examples of groups having D-finite cogrowth series for at least one finite generating set.
\end{abstract}

\tableofcontents

\section{Introduction}

Let $\Gamma$ be a group with finite generating set $S$, and let $S^{-1}$ be a distinct alphabet in bijection with $S$, where the bijection is given by the map $s\mapsto s^{-1}\in \Gamma$. Let $W$ be the set of words in $(S\cup S^{-1})^*$ that evaluate to the identity in $\Gamma$.
The cogrowth function $\psi\colon\N\longrightarrow \N$ is defined as $\psi(n)=\psi_n=\left|\{w\in W\mid |w|=n\}\right|$. The formal power series
\begin{align}
  \Psi_{\Gamma, S}(z) = \sum_{n=0}^\infty \psi_nz^n=\sum_{w\in W}z^{|w|}
\end{align} is called the {\em cogrowth series} of \(\Gamma\) with the set \(S\).

It is well known \cite{Cohen, Grig} that the \textit{cogrowth}, the exponential growth rate of \(\Psi_{\Gamma, S}(z)\), satisfies
\begin{align}
  \limsup_{n \longrightarrow \infty} \psi_n^{1/n}=2|S|
\end{align} if and only if $\Gamma$ is amenable. 

In this work, we do \textit{not} consider only freely reduced words; rather, we define the cogrowth series in terms of words that are equivalent to the identity, whether they are freely reduced or not. This is done to simplify much of the combinatorial analysis that follows. There are simple substitutions to move between the cogrowth series defined on all words and on freely reduced words \cite{Kukrational}.

The following well-known theorem by Kouksov \cite{Kukrational} classifies group presentations that have rational cogrowth series.

\begin{theorem}[Kouksov \cite{Kukrational}] The cogrowth series for a group with presentation $\langle S\mid R\rangle $, $|S|$ finite, is rational if and only if the group is finite.
\end{theorem}

An analogous result for algebraic cogrowth series is open, although a longstanding conjecture exists in the literature.
\begin{conj}[Folklore, mentioned in {\cite[page 1]{bishopDiagonal}}, {\cite[page 1]{bodart24}}]
\label{conj:algIFFvirtfree}
The cogrowth series for a group with presentation $\langle S\mid R\rangle $, $|S|$ finite, is algebraic if and only if the group is virtually-free.
\end{conj} 
A group is said to be virtually-free if it has a finite-index subgroup that is free. More generally, a group is virtually-\(P\) (for some property \(P\)) if it has a finite-index subgroup that has property \(P\).

This conjecture has been around in different guises for some time (see, for example, \cite{GLalley}), usually phrased in terms of Green's functions. Recall that the Green's function is the probability generating function $G(x, y; z) = \sum_{n\geq 0} P^{(n)}(x, y) z^n$, where $P^{(n)}(x, y)$ is the probability that a simple random walk starting at $x$ ends at $y$ after $n$ steps in the Cayley graph of the associated group (see \cite{woess}). Hence, \(G(e, e; z)\) can be obtained from the cogrowth series via the substitution $z\mapsto z/(2|S|)$.

The backwards direction of Conjecture~\ref{conj:algIFFvirtfree} follows from the combination of two notable theorems: Muller and Schupp (relying on a result of Dunwoody) show that a finitely generated group is virtually-free if and only if the set $W$ is a context-free language \cite{Dunwoody85, ms83}, and in particular the context-free grammar describing the word problem is unambiguous; and Chomsky and Sch\"utzenberger's result that unambiguous context-free languages have algebraic generating functions \cite{ChomS}. 

A standard superset of algebraic generating functions are D-finite generating functions. A generating function is said to be \textit{D-finite} (or \textit{holonomic}) if it satisfies a linear differential equation (of finite order) with polynomial coefficients (see \cite{Stanley2} for a proof of the inclusion). A sequence is said to be \textit{P-recursive} if it satisfies a linear recurrence relation with polynomial coefficients. A standard result (see \cite{Stanley}) states that an integer sequence $\psi_n$ is P-recursive if and only if its generating function $\sum_{n=0}^\infty \psi_n z^n$ is D-finite.

The cogrowth series for Baumslag-Solitar groups, $BS(M,N)=\langle a,t\mid ta^Mt^{-1}=a^N\rangle$, was studied in \cite{ERRW}. It was shown that when $M=N$, the cogrowth series is D-finite (and not algebraic for $N\leq 10$) with respect to the generating set $\{a^{\pm 1}, t^{\pm 1}\}$. It was also conjectured that for $M\neq N$ the series should be non-D-finite.

Garrabrant and Pak \cite{GPak} proved that any amenable linear group of exponential growth has non-D-finite cogrowth series with respect to any generating set. This includes for example the solvable Baumslag-Solitar groups $BS(1,N)=\langle a,x\mid xax^{-1}=a^N\rangle$. They also show that $F_{11}\times F_3$ (where $F_k$ denotes the free group of rank $k$) has D-finite cogrowth series for one finite generating set
but not for another.

Later, Bell and Mishna \cite{BellMishnaAmenable} extended this to any amenable group of super-polynomial (intermediate and exponential) growth by finding necessary conditions on the generating function to be D-finite. They also claimed that any virtually abelian group will have D-finite cogrowth series for any finite generating set, which was proved by 
Bishop \cite{bishopDiagonal} who showed more generally that if a group contains \(\Z^d \times F_k\) as a finite index subgroup (where $\Z^d$ is the free abelian group of rank $d$), then its cogrowth series is the diagonal of a rational function for all presentations, answering a question in \cite{GPak} and extending some results from \cite{ERRW}.

However, the exact cogrowth series is not known for many group presentations. Kuksov \cite{Kukseries} computed the cogrowth series for certain free products of free and finite groups using combinatorial and series arguments. Bell, Liu and Mishna \cite{bell2021cogrowth} construct a grammar via which they can compute the cogrowth series for free products of finite groups. Darbinyan, Grigorchuk and Shaikh \cite{GrigGCC} compute the cogrowth series of finitely generated subgroups of free groups via ergodic automata. Some more work in a similar vein includes generalisations of the cogrowth series \cite{grigMultiCogrowth, Humphries}, and computational work to estimate the cogrowth of Thompson's group \cite{ThompsonCogrowthERW,PriceGuttmann}. 

Additionally, not a lot is known about the groups that are not virtually-free. In the few known results, it is seen that the cogrowth series are either non-D-finite, as in \cite{GPak}, or D-finite but not algebraic, as in \cite{ERRW}. 
Pak and Sokup \cite{pak2022} show that even for virtually-nilpotent groups, the problem of writing the cogrowth series as a diagonal of a rational function is not computable. More recently, Bodart \cite{bodart24} gave an explicit virtually-nilpotent group presentation with a generating multiset where the cogrowth series is not D-finite.

The goal of the present article is to expand the list of group presentations for which the cogrowth series can be computed exactly, and in particular those whose cogrowth series is D-finite and not algebraic. Specifically, in this article, we first study the cogrowth series in the following family of presentations
\begin{equation}
\langle a_1,\dots, a_k \mid a_1^{p_1}=\dots=a_k^{p_k} \rangle,
\end{equation}
where $k \geq 2, p_i \geq 2$ are positive integers.

In Section~\ref{sec:other_cases} and the Appendix, we apply our techniques with mild generalisations to the following presentations of the 3-strand braid group $B_3$:
\begin{equation}\label{eq:B3preses}
\langle a,b \mid aba=bab \rangle \qquad 
\langle a,x \mid axa=x^2 \rangle
\end{equation} (see Section~\ref{sec:other_cases} to see how these present the same group). In each case, we prove the cogrowth series is again D-finite and non-algebraic. Another presentation of \(B_3\), \(\langle x,c \mid x^3=c^2 \rangle\), is covered by our main results, and we also give explicit asymptotics of the cogrowth series in the Appendix.

\subsection{Notation}
We use the standard notations for groups and presentations as in \cite{LS}. We write $ u\equiv v$, where \(u,v\) are words in \((S \cup S^{-1})^{*}\), to mean that the products $u$ and \(v\) are equal in the group $\Gamma$, and $u=g$ if the product $u$ is equal to the group element \(g \in \Gamma\). The length of the word $u$ is denoted $|u|$.

The (right-)Schreier graph of a group $\Gamma$ with respect to a finite generating set $S$ and subgroup $H$ is the graph whose vertices are the right-cosets of \(H\), \(\{H g \mid g \in \Gamma\}\), with directed edges \((Hg, Hgs)\) labeled $s$ for each coset $Hg$ and each $s\in S \cup S^{-1}$. In the case the $H$ is a normal subgroup, the Schreier graph of $\Gamma$ identical to the Cayley graph of the quotient group $H\backslash \Gamma$.\label{Schreier-normal-isometric}

Much of the paper focuses on the following family of groups.
 
\begin{defn}\label{defn:Gk}
  Let \(p_1, \ldots , p_k\) be integers at least \(2\) and $k\geq 2$. Let \(\mathcal{G}(p_1, \ldots , p_k) \) be the group presented by \(\langle a_1, \ldots, a_k \mid a_1^{p_1} = a_2^{p_2} = \cdots = a_k^{p_k} \rangle\). For simplicity of notation, we will often write \(\mathcal{G}_k\) in place of \(\mathcal{G}(p_1, \ldots p_k)\) when the values $p_i$ are understood.
\end{defn}

\subsection{Overview}
The main work of this paper is concerned with showing that the group $\mathcal G_k$ with respect to the generating set in Definition~\ref{defn:Gk} has a D-finite but non-algebraic cogrowth series. We will also show that the cogrowth is an algebraic number. We will often work with fixed \(p_1,\dots,p_k\) and simply denote the group by \(\mathcal{G}_k\).

Our method is roughly as follows. We find a central element $\Delta$ (that is, commutes with every element of the group), together with a normal form $\Delta^m w$. When multiplication by generators on the right does not change the normal form too dramatically, we come up with a finite system of two-variable generating functions counting various sets of words. We start with the set of words $W$ which evaluate to an element in $\langle \Delta \rangle$, so have normal form $\Delta^m$ for some integer $m$. We let 
\begin{align}
  F_k(z; q)=\sum_{m\in\Z}\sum_{w\in W} z^{|w|}q^m=\sum_{m\in\Z}\sum_{n\in\N} f_{n,m}z^nq^m
\end{align}
be the generating function for the function $f_{n,k}$ counting words in $W$ of length $n$ and having normal form $\Delta^m$.

By defining several related series, we are able to prove that $F_k(z; q)$ is algebraic. We can then obtain the cogrowth series by computing the diagonal $[q^0]F_k(z; q)$, which is then D-finite since the diagonal of a D-finite series is again D-finite \cite{MR929767}. 

We then show that the series is not algebraic by showing that it has asymptotic form $f_{n, 0} \sim C \mu^n n^{-j}$ where $j$ is a positive integer. 
The key here is to essentially decompose our generating function into a product of two ``easier'' ones. 
\begin{itemize}
  \item The first one counts paths in a tree or tree-like graph and usually has an $\mu^n n^{-3/2}$ term in its asymptotic formula, typical of algebraic generating functions. This will be our Schreier graph with finite tree-width.
  \item The second one is essentially a one-dimensional random walk on \(\Z\), which counts the exponent \(\Delta^m\). This gives us Gaussian behaviour, with its asymptotics of the form \(n^{-1/2} \exp (-C m^2/n)\).
\end{itemize}
So, together, we get $n^{-2}$ (or \(n^{-1}\) in edge cases); this proves the series cannot be algebraic. This is in contrast with \cite{ERRW}, where non-algebraicity could only be proved for a finite number of cases.

In Section~\ref{sec:special_case}, we compute the asymptotics explicitly for the special case with the presentation \(\langle a_1, \ldots, a_k \mid a_1^2 = a_2^2 = \cdots = a_k^2 \rangle\) by first counting walks on a \(k\)-regular tree, and then keeping track of the exponent \(\Delta^m\). This yields Theorem~\ref{thm:k_tree_expansion}.

For the general case, with presentation \(\langle a_1, \dots, a_k \mid a_1^{p_1} = a_2^{p_2} = \dots = a_k^{p_k} \rangle\), we need more work. 
Section~\ref{sec:gen_funcs_walks} begins by discussing the notion of ``winding number'', which we use to keep track of \(\Delta^m\). We then discuss how to build a system of generating functions for walks on this group. 

Unfortunately, computing the asymptotics in the general case explicitly is not feasible (though we do have numerical results for special cases in Appendix~\ref{appx:numerics}). Hence, we rely on general theorems in the literature \cite{scalinglimits22, drmota1997, drmota2009random}. We get the required asymptotics in two main parts. 

\begin{itemize}
 \item In part 1 (Section~\ref{sec:asymp1}), we get an asymptotic expansion for certain parts of our system. We do this by using the standard results, which we also state. This yields Theorem~\ref{thm:lll}, and we will use it later.
\end{itemize}
  
In order to prove Theorem~\ref{thm:lll}, we need to deal with two technical issues. Firstly, we need to find bounds on the variance of the winding number in a random walk. We do that in Section~\ref{sec:pat_dec_var}. In particular, we find a decomposition of the walk (for an upper bound) in Section~\ref{sec:decomp}, and prove a pattern theorem in Section~\ref{sec:pat}. Secondly, we need to deal with contributions from multiple singularities and saddle points depending on the parity of our system. We do that via generating function transformations in Section~\ref{sec:transform}.

\begin{itemize}
 \item In part 2 (Section~\ref{sec:asymp2}), we prove an asymptotic expansion for the entire system using tools from geometric group theory and probability on groups. In particular, we show that a closed random walk on our Schreier graph spends very little time at the root, which implies that the asymptotics of the entire system are similar to the asymptotics of the partial system proved earlier.
\end{itemize}
We then exploit the fact that our Schreier graph has finite tree-width to apply the results of Gouëzel~\cite{hyperbolic} to determine the asymptotics.

\section{Groups with presentation \(\langle a_1, \dots, a_k \mid a_1^2 = a_2^2 = \dots = a_k^2 \rangle\)}
\label{sec:special_case}
In this section, we consider the special case of Definition~\ref{defn:Gk} where $p_1=\dots=p_k=2$ with \(k \geq 2\), which will introduce many of the ideas needed for the general case. 
Thus we have 
\begin{align}
  \mathcal{G}(2,\dots,2) = \mathcal{G}_k = \langle a_1, \ldots, a_k \mid a_1^2 = a_2^2 = \cdots = a_k^2 \rangle.
\end{align} 

Let $\Delta=a_1^2 \equiv a_2^2 \equiv \dots \equiv a_k^2$. 
\begin{lemma}
 Each element of $\mathcal{G}_k$ can be written in normal form
 \begin{align}
 w &= \Delta^m v 
 \end{align}
 where $m\in \Z$ and the suffix $v\in\{a_1,\ldots ,a_k\}^*$ that does not contain $a_i^2$.
\end{lemma}
\begin{proof}
Let $u$ be a word in the generators and inverses. Replace $a_i^{-1}$ by $\Delta^{-1}a_i$ and $a_i^2$ by $\Delta$ everywhere in $u$ starting from left to right. Since $\Delta$ is central, commute all powers of $\Delta$ to the left. If this produces more positive powers, repeat.
\end{proof}

From this, we can construct the Schreier graph of $\mathcal{G}_k$ with respect to the generating set $\{a_1,\dots, a_k\}$ and subgroup $\langle \Delta \rangle$, which is a \(k\)-valent tree. Figure~\ref{fig abc cosests} shows an example for \(k=3\). The vertices of this graph are in bijection with the suffixes in the above lemma, and so we simply label them by the suffix. So, two vertices labeled by suffixes $v_1, v_2$ are connected by an edge labeled $s \in S \cup S^{-1}$ when $v_2 = v_1s$.

We note that at times it will be convenient to think of there being two directed edges between these vertices, one labelled by \(s\) and the other by its inverse.

\begin{figure}[h]
 \centering
 \includegraphics[width=0.5\textwidth]{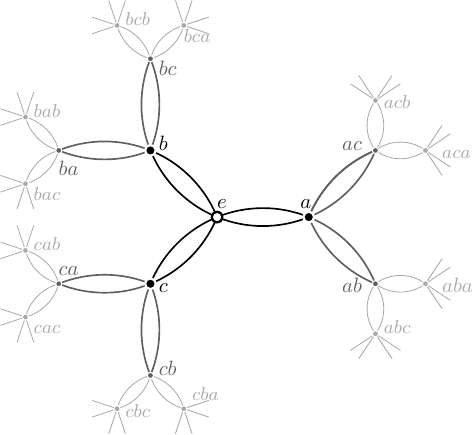}
 \caption{The Schreier graph of $\mathcal{G}_3 = \langle a,b,c \mid a^2=b^2=c^2 \rangle$; vertex $v$ corresponds to the coset $\langle \Delta \rangle v$ with the edges drawn by two arcs. To avoid overuse of subscripts, we have chosen generators $a,b,c$ rather than $a_1,a_2,a_3$.}
\label{fig abc cosests}
\end{figure}

\subsection{Walks in the \(k\)-regular tree}
The problem of counting words in $\mathcal{G}_k$ which are equivalent to the identity will be closely related to the problem of counting walks in the \(k\)-tree that start and end at the origin. Before we proceed with the cogrowth series, let's recall one way in which to count such walks.

\begin{defn}
Let $F_k(z)$ be the generating function that counts all walks that start and end at the root vertex of the $k$-valent tree. Similarly, define $A_i(z)$ to be the generating function that counts all walks that start and end at the vertex $a_i$, without ever visiting the root (one can look at the tree depicted in Figure~\ref{fig abc cosests}).
\end{defn}

We remark that the generating functions \(A_i(z)\) defined here are \textit{slightly} different from the ones in the general case defined in Section \ref{sec:gen_funcs_walks}. In the special case we are in, these ones are easier to work with explicitly; however, the general argument will also give us the same final \(F_k\).

We establish a system of algebraic equations satisfied by these generating functions via a decomposition of the walks that they count. The key idea is to decompose a walk counted by (say) $F_k(z)$ by cutting it when it last leaves the root vertex. This shows that any such walk is either trivial or is comprised of two shorter walks, one counted by $F_k$ and the other by one of the $A_i$s. This sort of factorisation is quite standard.

\begin{lemma}
 The above generating functions satisfy the following system of equations:
 \begin{subequations}
 \begin{align}
 A_1(z) &= A_2(z) = \cdots = A_k(z), \\
 A_1(z) &= 1 + (k-1) z^2 A_1(z)^2, \\
 F_k(z) &= 1 + k z^2 F_k(z) A_1(z)
\end{align}
\end{subequations}
and hence
\begin{align}\label{eqn kreg}
 F_k(z) &= \frac{2(k-1)}{k-2 + k\sqrt{1-4(k-1)z^2}}.
\end{align}
\end{lemma}
\begin{proof} The equality of the \(A_i\)s follows from symmetry. Now, consider a walk counted by $F_k$. Such a walk 
\begin{itemize}
 \item has length $0$, or
 \item can be decomposed by cutting it at the \textit{last} time it \textit{leaves} the root vertex.
\end{itemize}
In this second case, the walk decomposes into another walk counted by $F_k(z)$, a step away from the root, a walk counted by one of the $A_i$s, and finally a step back to the root vertex. Translating this into operations on the corresponding generating functions, we have.
\begin{align}
 F_k(z) &= 1 + F_k(z) \cdot z \cdot \sum_i A_i(k) \cdot z.
\end{align}

Now take a walk counted by $A_1$ and decompose it in a similar way. It either
\begin{itemize}
 \item has length $0$, or
 \item can be decomposed by cutting it at the \textit{last} time it \textit{leaves} the vertex \(a_1\).
\end{itemize}
In the second case, the walk decomposes into a walk counted by $A_1(z)$, a step away from $a_1$ to the vertex corresponding to \(a_1 a_i\) for \(i \neq 1\), a walk counted by $A_i(z)$ for $i \neq 1$ (respectively), and a final step back to $a_1$. Hence we obtain
\begin{align}
 A_1(z) &= 1 + A_1(z) \cdot z \cdot \sum_{i \neq 1}A_i(z) \cdot z,
\end{align}
which gives us the required system of equations.

The equation for $A_1(z)$ has two possible solutions
\begin{align}
 A_1(z) &= \frac{1 \pm \sqrt{1-4(k-1)z^2}}{2(k-1)z^2}.
\end{align}
The positive branch is $\mathcal{O}(z^{-2})$ and so is not a valid generating function, while the negative branch is valid and its first few terms are $1 + (k-1)z^2 + \mathcal{O}(z^4)$.
Substituting the negative branch into the equation for $F_k(z)$ gives the result.
\end{proof}

We demonstrate that one can actually write out the coefficients of this generating function quite explicitly.
\begin{lemma}
For the above system, the generating function \(F_k(z)\) is
\begin{align}
 F_k(z) &= 1 + k \sum_{n \geq 1} z^{2n} \sum_{m=0}^{n-1} \frac{n-m}{n} \binom{2n}{m} (k-1)^m.
\end{align}
\end{lemma}

\begin{proof}
To obtain the explicit expansion of $F_k(z)$, start by rewriting
\begin{align}
 F_k(z) &= \frac{1}{1-k z^2 A_1(z)} = \sum_{j \geq 0} k^j z^{2j} A_1(z)^j, \\
 [z^{2n}]F_k(z) &= \sum_{j=0}^{n} k^j [z^{2n-2j}] A(z)^j = \sum_{\ell=0}^n k^{n-\ell} [z^{2\ell}]
A(z)^{n-\ell}
\end{align}
where we have used the fact that $A(z)$ contains only even powers of $z$. Notice that $A(z)$ is very nearly the Catalan generating function and its expansion is well understood:
\begin{align}
 [z^{2n}] A(z)^\ell &= \frac{\ell}{n+\ell}\binom{2n+\ell-1}{n} (k-1)^n.
\end{align}
Hence, some standard manipulations give
\begin{align}
 [z^{2n}]F_k(z) &= \sum_{\ell=0}^n k^{n-\ell} \cdot \frac{n-\ell}{n} \binom{n+\ell-1}{\ell} (k-1)^\ell \\
 &= \frac{k}{n} \sum_{m=0}^{n-1} (k-1)^m (n-m) \binom{2n}{m}.
\end{align}
Along the way, one uses the equality
\begin{align}
 \binom{2n}{m} &= \sum_{\ell=0}^m \binom{n+\ell-1}{\ell} \binom{n-\ell}{m-\ell},
\end{align}
which is fairly standard.
\end{proof}

\subsection{Back to cogrowth}

In order to compute the cogrowth series, we will use a very similar decomposition. Indeed, the main difference is that as we move around the Schreier graph, we must also keep track of powers of $\Delta$. Let us start by forming a generating function analogous to $F_k$ above.
\begin{defn}
For $m\in\Z$ let $W_m$ be the set of words in $\{a_i, a_i^{-1}\}_{i \leq k}$ that are equivalent to group elements whose
normal form is exactly $\Delta^m$, and then let $W = \bigcup_m W_m$. Hence, a word is in $W$ if and only if it corresponds to a group element lying in the subgroup generated by $\Delta$. Then define
\begin{align}
 F_k(z;q) &= \sum_m \sum_{w \in W_m} z^{|w|} q^m
\end{align}
where $|w|$ denotes the length of $w$. This counts all words equivalent to elements in the coset $\langle \Delta \rangle$ according to their length and also the exponent of $\Delta$. The cogrowth series is then the constant term of $F(z,q)$ with respect to $q$.

Hence, we note that \([z^n q^m] F_k = \#\) of words of length \(n\) equivalent to \(\Delta^m\).

\end{defn}
To obtain this generating function, we first need to define some auxiliary functions (analogous to the $A_i$s above).

\begin{defn}
Let $W_{i,m}$ be the subset of $W_m$ so that if $w = uv$ then $u \not \equiv \Delta^j a_i$ for any $j$. Further, let $W_i = \bigcup_m W_{i,m}$. These words describe walks in the Schreier graph that start at the origin and never visit $\langle \Delta \rangle a_i $. Let their generating function be
\begin{align}
 A_i(z;q) &= \sum_m \sum_{w \in W_{i,m}} z^{|w|} q^m.
\end{align}
\end{defn}

We find a set of equations satisfied by these generating functions using a similar decomposition.
\begin{lemma}
\label{lem:k-tree}
 The generating functions $A_i(z;q)$ and $F_k(z;q)$ satisfy the following set of equations:
 \begin{subequations}
 \begin{align}
 A_1(z;q) &= A_2(z;q) = \cdots = A_k(z;q), \\
  A_1(z;q) &= 1 + (k-1) z^2 (q+2+q^{-1}) A_1(z;q),\\
 F_k(z;q) &= 1 + k z^2 (q+2+q^{-1}) F_k(z;q) \cdot A_1(z;q).
\end{align}
\end{subequations}
\end{lemma}
\begin{proof}
The first equation follows from symmetry. For the remaining ones, we proceed as we did above. We decompose any walk counted by $F_k$ by cutting it when it last leaves the coset $\langle
\Delta^k \rangle$. Any walk counted by $F_k$
\begin{itemize}
 \item has length $0$, or
 \item can be decomposed by cutting it at the last time it leaves the coset $\langle \Delta \rangle$.
\end{itemize}
In this second case, the walk, $w$, decomposes into a shorter walk counted by $F_k$, a step away from $\langle \Delta
\rangle$ to one of the cosets $\langle \Delta \rangle a_i$, then another walk, $\hat{w}$, and finally a step back to $\langle \Delta \rangle$. 

Say our walk has stepped to $\langle \Delta \rangle a_j$, then the intermediate walk $\hat{w}$ must start and end at $\langle \Delta \rangle a_j$, while never visiting $\langle \Delta \rangle$. These are precisely the walks counted by $A_j(z;q)$.

Now, in order to step from $\langle \Delta \rangle$ to $\langle \Delta \rangle a_j$, we must either append an $a_j$ or $a_j^{-1}$.
\begin{itemize}
 \item Appending $a_j$ does not change the power of $\Delta$.
 \item Appending $a_j^{-1}$ decreases the power of $\Delta$ since $\Delta^m a_j^{-1} = \Delta^{m-1} a_j$.
\end{itemize}

Similarly, when we step back, we must also append an $a_j$ or $a_j^{-1}$.
\begin{itemize}
 \item Appending $a_j$ increases the power of $\Delta$, since $\Delta^m a_j^2 = \Delta^{m+1}$.
 \item Appending $a_j^{-1}$ does not change the power of $\Delta$ since $ \Delta^m a_j a_j^{-1} = \Delta^m$.
\end{itemize}
This gives us
\begin{align}
 F(z;q) &= 1 + F(z;q) \cdot z (1+q^{-1})\cdot \sum_j A_j(z;q) \cdot z (1+q).
\end{align}

A very similar argument for decomposing \(A_1\) gives us
\begin{align}
 A_1(z;q) &= 1 + A_1(z;q) \cdot z(1+q^{-1}) \cdot \sum_{j \neq 1} A_j(z;q)\cdot z (1+q).
\end{align}
These equations can now be simplified to get the result.
\end{proof}

Solving this system of equations then gives us
\begin{theorem}[Cogrowth series for $\mathcal G(2,2,\dots, 2)=\mathcal G_k$ with $k\geq 2$]
\label{thm:k_tree_expansion}
 The generating function $F_k(z;q)$ for $k\geq 2$ is given by
\begin{align}
 F_k(z;q) &= \frac{2(k-1)}{k-2 + k\sqrt{1-4(k-1)(q+2+q^{-1})z^2}}.
\end{align}
Hence, the cogrowth series of the group $\mathcal{G}_k$ is given by
 \begin{align}
 [q^0] F_2(z;q) &= 1 + \sum_{n \geq 1} z^{2n} \binom{2n}{n}^2, \\
 [q^0]F_k(z;q) &= 1 + k \sum_{n \geq 1} z^{2n} \binom{2n}{n} \sum_{m=0}^{n-1} \frac{n-m}{n} \binom{2n}{m} (k-1)^m.
 & (k \geq 3)
\end{align}
Hence, 
\begin{align}
  [z^{2n} q^0] F_2(z;q) &= \frac{1}{\pi n}\cdot 4^{2n} \left( 1 + \mathcal{O}(n^{-1})\right), \\
 [z^{2n} q^0] F_k(z;q) &= \frac{k(k-1)}{(k-2)^2} \cdot \frac{ 4^{2n} (k-1)^n }{ \pi n^2} \left(1 + \mathcal{O}(n^{-1}) \right).
& (k \geq 3)
\end{align}
Consequently, the cogrowth is $4\sqrt{k-1}$, and the cogrowth series is D-finite but not algebraic for any integer $k \geq 2$.
\end{theorem}
\begin{proof}
Solving the system in Lemma~\ref{lem:k-tree} gives the required expression. We obtain the explicit expression for the coefficients by noting that $F_k(z;q)$ can be obtained from the generating function for walks on $k$-regular trees by substituting $z^2 \mapsto z^2(q+2+q^{-1})$. More precisely, let $T_k(z)$ be the generating function for walks on $k$-regular trees, which we called $F_k(z)$ in Equation~\eqref{eqn kreg}. Then
 \begin{align}
 F_k(z;q) &= T_k(z \sqrt{q+2+q^{-1}}), \\
 [z^{2n} q^0]F_k(z;q) &= [z^{2n}] T_k(z) \cdot [q^0] (q+2+q^{-1})^{n}, \\
 [z^{2n} q^0]F_k(z;q) &= [q^0] (q+2+q^{-1})^{n} \cdot k \sum_{m=0}^{n-1} \frac{n-m}{n} \binom{2n}{m} (k-1)^m,
\end{align}
and $[q^0] (q+2+q^{-1})^n = \binom{2n}{n}$.

We first consider \(k \geq 3\) and then turn to \(k=2\), which simplifies considerably. So for \(k \geq 3\), to obtain the asymptotics of the cogrowth series, we use standard tools from analytic combinatorics
\cite{Flajolet} to construct an uniform expansion of $T_k(z)$:
\begin{align}
 [z^{2n}] T_k(z) &=
 \frac{k(k-1)}{(k-2)^2} \cdot \frac{ 4^n (k-1)^n }{ \sqrt{\pi n^3} } \left(1 + \mathcal{O}(n^{-1}) \right)
\end{align}
and similarly find a uniform expansion of $\binom{2n}{n}$:
\begin{align}
 \binom{2n}{n} &= \frac{4^n}{\sqrt{n \pi}} \left(1 + \mathcal{O}(n^{-1}) \right).
\end{align}
The product of these expansions gives the asymptotics of the cogrowth series.

Finally, the function $F_k(z;q)$ is algebraic and hence D-finite (see \cite{stanley1980}, for example). D-finite series are closed under extraction of constant terms \cite{MR929767}, and hence the cogrowth series is D-finite. Finally, we know that the cogrowth series is not algebraic because of the dominant asymptotics. The $n^{-2}$ correction to the dominant exponential term is incompatible with an algebraic generating function (\cite[Section VII.7]{Flajolet} and \cite{jungen}).

For \(k=2\), note that the generating function simplifies to 
\begin{align}
  F_2(z;q) = \frac{1}{\sqrt{1-4z^2(q+2+q^{-1})}}
\end{align}
which can be explicitly expanded to get the required form of \([q^0]F_2(z;q)\). The uniform expansion of \(\binom{2n}{n}\) then gives us the result.
\end{proof}

This gives an infinite family of groups for which we can determine the cogrowth series explicitly and, further, demonstrate that the series is not an algebraic function.

\section{Groups with presentation \( \langle a_1, \dots, a_k \mid a_1^{p_1} = a_2^{p_2} = \dots = a_k^{p_k} \rangle\)}
\label{sec:gen_funcs_walks}

Let $\mathcal G_k=\mathcal G(p_1,\dots, p_k)$ where $k,p_i\geq 2$, and define \(\Delta = a_1^{p_1} \equiv a_2^{p_2} \equiv a_3^{p_3} \equiv \cdots\).

From here onwards, we assume at least one of the \(p_i\)s is at least \(3\) unless stated otherwise. Hence, \(\max\{p_1,\dots,p_k\} \geq 3\). We do this since the case with all \(p_i = 2\) has been handled in Section~\ref{sec:special_case}.

\begin{lemma}
\label{lem:normal}
The element \(\Delta\) is central. Consequently, \(\langle \Delta \rangle\) is a normal subgroup of \(\mathcal{G}_k\).
\end{lemma}
\begin{proof}
Since \(a_1^{p_1} \equiv a_i^{p_i}\) for all \(i\), it is enough to show that \(a_1^{p_1}\) commutes with every element of \(\mathcal{G}_k\). 

Consider a generator \(a_j\) of \(\mathcal{G}_k\). Then,
  \begin{align}
    a_j \cdot a_1^{p_1} = a_j \cdot a_j^{p_j} =a_j^{p_j+1} = a_j^{p_j} \cdot a_j = a_1^{p_1} \cdot a_j.
  \end{align}
Apply this to any word in the generators and their inverses to get the result.
\end{proof}
\begin{cor}
  Every element \(w \in \mathcal{G}_k\) can be written in the normal form 
  \begin{align}
    w = \Delta^m v \quad (m \in \Z)
  \end{align}
  where the suffix \(v \in \mathcal{G}_k\) is a word in
  strictly positive powers of the generators \(\{a_1, \ldots, a_k\}\) such that it doesn't contain any subword from \(\{a_1^{p_1} , \ldots , a_k^{p_k}\}\).
\end{cor}

Consider the Schreier graph of \(\mathcal{G}_k\) with respect to the generating set $\{a_1,\dots, a_k$ and subgroup \(\langle \Delta \rangle\). As before, since cosets are in bijection with suffixes as in the previous corollary, we may assume vertices are labelled by suffixes.

An example of \(\mathcal{G}_2 = \langle a, b \mid a^3 = b^4 \rangle\) is shown in Figure \ref{fig:three_four_graph}. Note again that we use \(a, b\) for generators instead of \(a_1, a_2\) to avoid overuse of subscripts.

\begin{figure}[h]
\centering
  \includegraphics[width=0.5\linewidth]{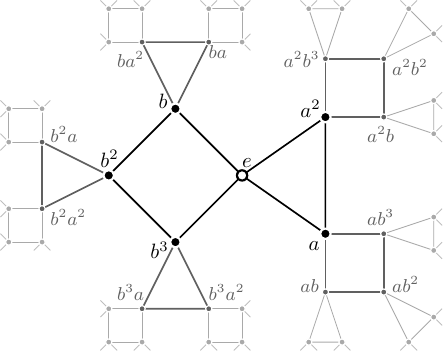}
    \caption{The Schreier graph of $\mathcal{G}_2 = \langle a, b \mid a^3 = b^4 \rangle$; vertex $p$ corresponds to the coset $\langle \Delta \rangle p$.}
    \label{fig:three_four_graph}
\end{figure}

\begin{defn}
  For \(m \in \Z\), define \(W_m\) to be the set of words in \(\{a_1, \ldots, a_k, a_1^{-1}, \ldots a_k^{-1}\}\) such that their normal form is exactly \(\Delta^m\).

  Then define the generating function \(F_k\) is:
  \begin{align}
  F_k(z; q) = \sum_{m} \sum_{w\in W_m} z^{|w|} q^m.
\end{align}
Hence, the coefficient \([z^n q^m] F_k(z;q) \) is the number of words in \(\mathcal{G}_k\) with length \(n\) and normal form \(\Delta^m\). Hence, the cogrowth series is given by \([q^0]F_k\).
\end{defn}

We now describe a system of equations for \(F_k(z; q)\) based on walks on the Schreier graph.

\subsection{Winding number}
To keep track of the power of \(q\) in our generating functions, we define a notion of a \textit{winding number} for a closed walk on our graphs. The winding number will track the power of \(\Delta\) in the normal forms of group elements visited by the walk as it moves around the group.

First consider the case in which \(p_i \geq 3\), such as in the square and triangle adjacent to \(e\) depicted in Figure~\ref{fig:three_four_graph}. We will consider the \(p_i=2\) case shortly. A walk starting at \(e\) and moving anticlockwise around the triangle constructs words with normal forms \(e \mapsto a \mapsto a^2 \mapsto \Delta^1 \). Similarly, a walk moving clockwise around the triangle constructs \(e \mapsto a^{-1} \mapsto a^{-2} \mapsto \Delta^{-1}\). We can use \(q)\) to keep track of that change in the exponent of \(\Delta\) by accumulating a weight of \(q\) when we step from \(a^2\) to \(e\) and \(q^{-1}\) when we make the reverse step. We do similarly when between \(b^3\) and \(e\).

More generally, we will consider the graph to be embedded into the plane so that each facet is oriented in this way; positive powers of \(a_i\) will move anticlockwise around the face, while negative powers move clockwise.

Note that whenever a walk enters a facet moving away from the root, we must exit from the same vertex going back to the root since our graph is a ``tree of polygons''. This step can either use the same edge we took after entering the vertex, or a different edge. If we use the same edge in both the directions then this gives weight \(1 \cdot 1\) or \(q^{-1} \cdot q\) and so does not result in a net change in the exponent of \(q\). On the other hand, if the walk enters and leaves by different edges, then it gains \(1 \cdot q\) if it moves anticlockwise about the facet, or \(q^{-1}\cdot 1\) if it moves clockwise about the facet. See Figure~\ref{fig:winding}. By keeping track of these facet traversals, as the winding number, we keep track of the exponent of \(q\) and so of \(\Delta\) in the normal form. 

Now turn the case of \(p_i=2\). Starting at \(\Delta^m w\), multiplying by \(a_i\) gives \(\Delta^m w a_i\), while multiplying by \(a_i^{-1}\) gives \(\Delta^{m-1} w a_i\). Similarly, starting at \(\Delta^m w a_i\) multiplying by \(a_i\) gives \(\Delta^{m+1} w\) while multiplying by \(a_i^{-1}\) gives \(\Delta^m w\).
Consequently, we draw a bigon to denote these possibilities and pick one of the edges and give it a weight \(q\) when traversed towards the root, and \(q^{-1}\) when traversed away from the root. Just as above, we can use this factor of \(q\) to denote the change in winding number, and so the power of \(\Delta\) in the normal form.

\begin{defn}  
  The \textit{winding number} for a closed walk is the signed number of loops we make across the facets of the graph. We use the convention that any clockwise loop contributes \(-1\), and any anticlockwise loop contributes \(+1\). Note that if a walk enters and leaves a facet by the same edge, then it contributes \(0\) to the winding number.
\end{defn}
In particular, if we go across a facet such that we end up at the same vertex we started with, using a \textit{different} edge than our starting edge, we add a \(+1\) or \(-1\) to the winding number (depending on the orientation). This doesn't have to be immediate; we can enter a facet, leave from a different vertex, and come back. This keeps the winding number preserved. Figure~\ref{fig:winding} shows some examples of winding number contributions.

\begin{figure}[h]
  \centering
  \includegraphics[width=0.8\linewidth]{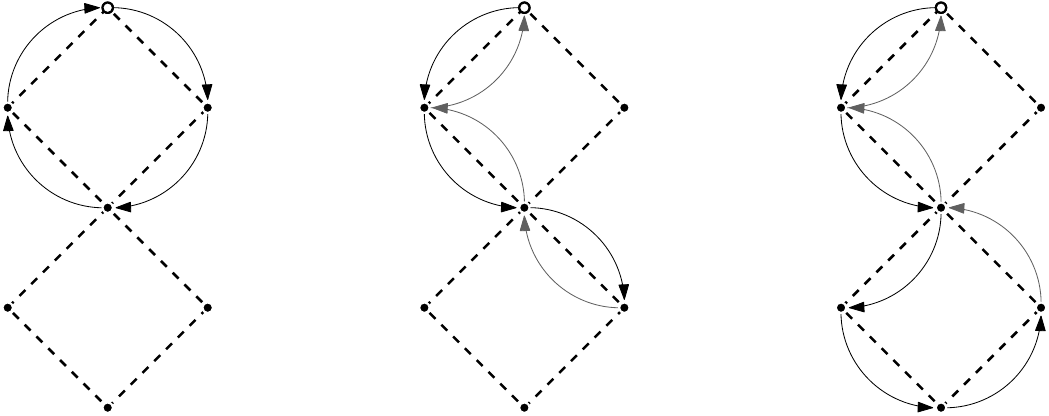}
  \caption{(Left to right): Walks starting and ending at the highlighted vertex giving winding number contributions of \(-1, 0, +1\) respectively.
  } 
  \label{fig:winding}
\end{figure}

We root the Schreier graph at the vertex corresponding to the identity. If we take closed walks on the graph, we see that the winding number of the walk is the same as the exponent of \(\Delta\) picked up for the corresponding word. This is similar to the argument we did earlier for the special case of the \(k\)-tree in Section~\ref{sec:special_case}. Hence, the \(F_k(z; q)\) defined above for words in the group can be written as
\begin{align}
  F_k(z; q) = \sum_{m} \sum_{w\in W_m} z^{|w|} q^m = \sum_{n} \sum_{m} f_{n, m} z^n q^m,
\end{align}
where \(f_{n, m}\) is the number of closed walks of length \(n\) with winding number \(m\).

To be precise, this shows that \(f_{n, m}\) is also the number of words of length \(n\) in the subgroup with normal form \(\Delta^m\). This is due to every closed walk being in direct bijection with such a word. 

\subsection{Generating functions for walks}

We now describe how to write a system of equations for \(F_k\). Rather than doing this in full generality, we demonstrate it for \(\mathcal{G}(3,4) = \langle a_1, a_2 \mid a_1^3 = a_2^4 \rangle\) first. The other cases follow similarly.

\begin{figure}[h]
  \centering
  \includegraphics[width = 0.8\textwidth]{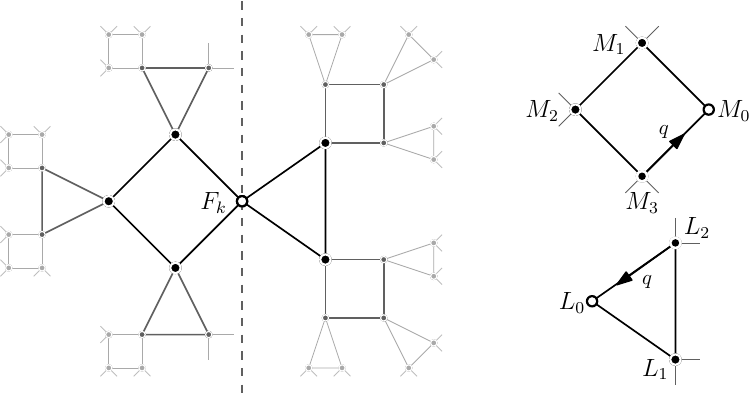}
  \caption{Building our system from two different facets. The dashed line represents where we ``cut'' our graph to make two one-sided graphs. We have highlighted the edges that accumulate a weight of \(q\) when traversed anticlockwise and \(q^{-1}\) when traversed clockwise.}
  \label{fig:twopolygon}
\end{figure}

\begin{defn}
Let $F_k(z; q)$ be the generating function that counts all closed walks that start and end at the root vertex of the graph. 
\end{defn}

Note that it lies on the intersection of a triangle and a square. Consider the \textit{one-sided graph} corresponding to the triangle. That is, choose the subgraph in the direction of the triangle: these are vertices that are labelled by suffix words that do not start with \(b\). This is the graph on the right side of the dashed line in Figure~\ref{fig:twopolygon}.

\begin{defn}
  Consider the triangle and square that contain the root vertex in Figure~\ref{fig:twopolygon}; these are separated by a dashed line which splits the graph into two subgraphs joined at the root. 
  
  Define $L_0(z; q)$ to be the generating function that counts all walks that start and end at the root vertex, without ever visiting a vertex that is labelled by a suffix that starts with \(b\). These are the walks that do not cross the dashed line indicated in Figure~\ref{fig:twopolygon} and stay in the subgraph originating from the triangle away from the root.

  Moving anticlockwise around the triangle from the origin, label the vertices \(L_0\) (being the root), \(L_1\) ,and \(L_2\). Then define \(L_1(z; q)\) and \(L_2(z;q)\) to be the generating functions that counts all walks that start at the root vertex, and end at \(L_1\) and \(L_2\) (respectively) while staying in the subgraph originating from the triangle away from the root.
  
  Now consider the other subgraph; we can similarly define vertices \(M_0, M_1, M_2, M_3\). These in turn allow us to define generating functions \(M_i(z;q)\) for walks in the subgraph on the left of the dashed line in Figure~\ref{fig:twopolygon} ending at \(M_i\). These walks never visit any vertex that is labelled with a suffix starting with \(a\).
\end{defn}

\begin{lemma}
The above generating functions satisfy the system of equations:

\vspace{-1em}
\begin{minipage}{0.49\textwidth}
\begin{subequations}
  \begin{align}
  L_0 &= 1 + z[L_1 + q \cdot L_2], \\
  L_1 &= z [L_0 + L_2]\cdot M_0, \\
  L_2 &= z [q^{-1} \cdot L_0 + L_1]\cdot M_0.
  \end{align}
\end{subequations}
\end{minipage}
\begin{minipage}{0.49\textwidth}
\begin{subequations}
\begin{align}
  M_0 &= 1 + z[M_1 + q \cdot M_3], \\
  M_1 &= z [M_0 + M_2]\cdot L_0, \\
  M_2 &= z [M_1 + M_3]\cdot L_0, \\
  M_3 &= z [q^{-1} \cdot M_0 + M_2]\cdot L_0.
\end{align}
\end{subequations}
\end{minipage}
\begin{align}
  L_0 = \frac{1}{1-P_L}, && M_0 = \frac{1}{1-P_M}, && F_k = \frac{1}{1-(P_L+P_M)}.
\end{align}
\end{lemma}

\begin{proof}
  Consider a walk counted by \(L_0\). Such a walk:
  \begin{itemize}
    \item has length \(0\), or
    \item had its last step taken from either \(L_1\) or \(L_2\) to \(L_0\).
  \end{itemize}
  In the second case, we get a walk from \(L_0\) to \(L_1\) or \(L_2\), followed by a step back to \(L_0\). Translating this into generating functions, we get
  \begin{align}
    L_0(z; q) = 1 + (L_1(z;q) + q \cdot L_2(z;q)) \cdot z.
  \end{align}
  The factor of \(z\) comes simply from adding an extra step. Any walk ending at \(L_1\) has normal form \(\Delta^m a\), and so stepping to \(L_0\) requires us to multiply by \(a^{-1}\). This results in an element with normal form \(\Delta^m\) (hence no change in the power of \(q\)). Stepping from \(L_2\) to \(L_0\) requires us to multiply by \(a\), giving an element equivalent to \(\Delta^m a^2 a \equiv \Delta^{m+1}\). Hence, the power of \(q\) increases by 1.

  Now, consider a walk counted by \(L_1(z;q)\). Such a walk can be decomposed by cutting it at the \textit{last} time a step was taken to \(L_1\) from either \(L_0\) or \(L_2\). Note that the normal forms of the corresponding words do not change the power of \(\Delta\), so there is no factor of \(q\), but the power of \(z\) increases by \(1\) corresponding to the increase in length. 
  This subwalk is counted by \((L_0(z;q) + L_2(z;q))\cdot z\). The remainder of the walk stays in the subgraph at or below the vertex \(L_1\). That subwalk is counted by \(M_0(z;q)\) due to vertex transitivity. This gives the equation for \(L_1(z;q)\).

  We make a similar decomposition for a walk counted by \(L_2(z;q)\) by cutting it at the \textit{last} time a step was taken to \(L_2\) from either \(L_1\) or \(L_0\). Notice that when stepping from \(L_0\) to \(L_2\) the corresponding word changes \(\Delta^m \mapsto \Delta^m a^{-1} \equiv \Delta^{m-1} a^2\) and so decreases the power of \(q\) by 1. Hence we obtain a subwalk counted by \( (L_1(z;q) + q^{-1}L_0(z;q) ) z\). As per above the other subwalk is counted by \(M_0(z;q)\).
  
  Similar arguments give the equations for \(M_i(z;q)\). A little care is required for \(L_2(z;q)\) and \(M_3(z;q)\) to take into account the winding number as the walk traverses facets of the graph.

  For the remaining equations, we break these walks into primitive subwalks, which is a standard technique in enumeration (see \cite[Theorem~I.1]{Flajolet}). Any walk counted by \(L_0(z;q)\) (or \(M_0(z;q)\)) can be decomposed into a sequence of \textit{primitive walks}, where a primitive walk is a walk starting and ending at \(L_0\) (or \(M_0\)), but not returning to the root anytime except the start and the end. Let the primitives be counted by \(P_L(z;q)\) (or \(P_M(z;q)\)). Hence,
  \begin{align}
    L_0 = 1 + P_L + P_L^2 + \cdots = \frac{1}{1 - P_L}, \\
    M_0 = 1 + P_M + P_M^2 + \cdots = \frac{1}{1 - P_M}.
  \end{align}

  We can do the same thing with walks counted by \(F_k(z;q)\). Every time the walk returns to the root, the walk has arrived from either an \(L\)-vertex or an \(M\)-vertex and so is adding a primitive \(P_L\) or \(P_M\). This gives   
  \begin{align}
    F_k = 1 + (P_L+P_M) + (P_L+P_M)^2 + \cdots = \frac{1}{1 - (P_L+P_M)}.
  \end{align}
  which gives us the system.
  \end{proof}

\subsection{The system of equations for \(\mathcal{G}(p_1,\dots,p_k)\)}
\label{sec:eqn_general}
For the general case, we consider the group presentation \(\langle a_1 , \ldots , a_k \mid a_1^{p_1} = \cdots = a_k^{p_k} \rangle\). This gives us a Schreier graph consisting of facets which are polygons with \(p_i\) sides. For each such facet, we get a \textit{one-sided} graph originating from the root.

\begin{defn}
  Let \(L_{j}^{(i)}\) denote the generating function corresponding to the \(j\)-th vertex on the \(i\)-th facet. It counts the walks (weighted by the winding number) on the one-sided graph of the \(i\)-th facet starting at the root and ending at the \(j\)-th vertex.
\end{defn} 

As an example, for the earlier system with \(L_0, L_1, L_2, M_0, M_1, M_2, M_3\) we have \(L_0^{(1)} = L_0\), \(L_0^{(2)} = M_0\), \(L_1^{(1)} = L_1\), and so on. 

With this, we can write our system in the general case:
\begin{align}
  L_0^{(i)} &= 1 + z [L_1^{(i)} + qL_{p_i-1}^{(i)}], \\
  L_r^{(i)} &= z [L_{r-1}^{(i)} + L_{r+1}^{(i)}] \sum_{j \neq i} L_0^{(j)} && (1 \leq r \leq p_i - 2), \\
  L_{p_i - 1}^{(i)} &= z [L_{p_i-2}^{(i)} + q^{-1}L_{0}^{(i)}] \sum_{j \neq i} L_0^{(j)},
\end{align}
for all \(1 \leq i \leq k\). These govern the walks on one-sided graphs. For walks on the entire graph, we also require the following.
\begin{align}
  L_0^{(i)} = \frac{1}{1 - P_i} \quad (1 \leq i \leq k), && F_k = \frac{1}{1 - \sum_{i \leq k}P_i}. \label{eq:connect}
\end{align}

In Figure~\ref{fig:threepolygon} we show the centre of the Schreier graph for the case \(\mathcal{G}(3,4,5)\).
\begin{figure}[h]
  \centering
  \includegraphics[width=0.9\textwidth]{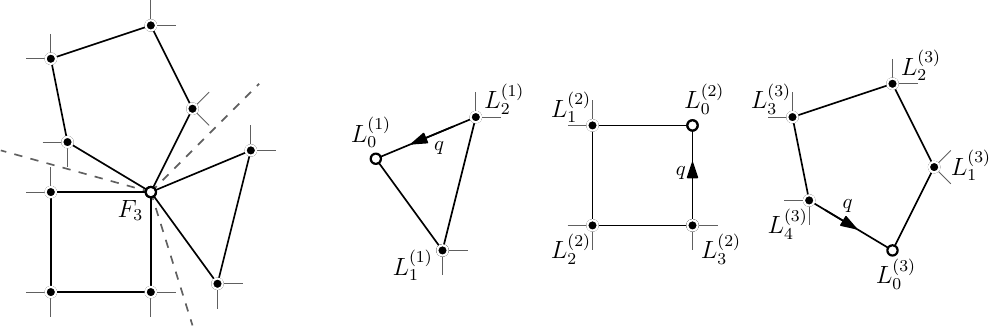}
 \caption{Building blocks for this Schreier coset graph and system of walks for \(\mathcal{G}(3,4,5)\). Again, the dashed lines on the left side represent the ``cuts'' we make to get our one-sided graphs.}
  \label{fig:threepolygon}
\end{figure}

We note here that even though we are using low values of \(p_i\)'s in the above discussion, we see that the system is already fairly unwieldy. In particular, the number of equations in the system is \(\sum_{i \leq k} p_i + k + 1\), of which \(\sum_{i \leq k}(p_i - 1)\) have \(2(k-1)\) quadratic terms each. When we explored these systems at even quite modest values of the \(p_i\), we found that computer algebra systems were not able to reduce the set of equations down by eliminating variables or compute a discriminant; the required memory and time appeared to grow extremely quickly with increasing \(p_i\)'s. We didn't explore this behaviour in detail.

However, for small values of the parameters, one can use computer algebra systems to compute the exact asymptotics. In particular, Appendix~\ref{appx:torus_knot} shows the numeric computations for the torus knot group (\(k=2\), \(p_1 = 2\), \(p_2 = 3\)).

We note that we could generalise the above to keep track of winding number contributions from different facet types. The resulting system is very similar, and the asymptotics can be computed similarly with a little more bookkeeping. We do not need that level of granularity for the results of this paper.

\section{Part 1: Asymptotics of \(L_0\)}
\label{sec:asymp1}

We are now ready to proceed to the asymptotics of the solutions of these systems. For technical reasons we break this into two parts. Here, in the first part, we deal with the asymptotics of generating functions of walks on one-sided subgraphs. Since the generating function of all walks depends on that of the walks on the one-sided subgraphs, but not vice versa, we cannot jump straight to the asymptotics of walks on the full graph. Those asymptotics are discussed in 
the next part, Section~\ref{sec:asymp2}.

\subsection{General results on asymptotics}
In this section, we talk about some well-known results about the asymptotics of algebraic functions. We state the relevant results from \cite{scalinglimits22, drmota1997, drmota2009random} here without proof. The proofs and more general analysis of the ``typical behaviour'' of algebraic generating functions can be found in these references, as well as in \cite[Section~VII.7]{Flajolet}.

The following theorem is a part of \cite[Theorem~A.6]{scalinglimits22}. Similar to the treatment in \cite{scalinglimits22} and \cite{drmota2009random}, we suppress the dependence on \(q\) for now since it plays no role in this result.

\begin{theorem}[Bassino et al. \cite{scalinglimits22}]
\label{thm:eigenvalue}
  Consider a system of equations denoted as \(\bm{Y}(z) = \bm{\Phi}(z, \bm{Y}(z))\), where \(\bm{\Phi}(z, \bm{y})\) is a vector of polynomials in \(z\) and \(\bm{y}\) over \(\Z\). Let \(\mathbb M\) be the Jacobian of \(\bm{\Phi}(z, \bm{y})\) wrt \(\bm{y}\). Also, assume the following:
\begin{enumerate}
  \item \(\bm{\Phi}\) is strongly connected.
  \item \(\bm{\Phi}(0, \bm{0}) = \bm{0}\).
  \item \(\mathbb M (0, \bm{0}) = \begin{bmatrix}
    0
  \end{bmatrix}\) as a matrix.
  \item \(\bm{\Phi}(z, \bm{0}) \neq \bm{0}\).
  \item \(\bm{\Phi}\) is not linear in the second argument.
\end{enumerate}
Then there is a unique solution \(\bm{Y}(z)\) of the equation \(\bm{Y}(z) = \bm{\Phi}(z, \bm{Y}(z))\) in the ring of formal power series with no constant coefficient. All entries have non-negative coefficients and the same radius of convergence \(\rho < \infty\). 

Additionally, there exists \((z, \bm{y})\) in the region of convergence of \(\bm{\Phi}\) such that \(\bm{y} = \bm{\Phi}(z, \bm{y})\), and \(\mathbb M(z, \bm{y})\) has dominant eigenvalue \(1\).
\end{theorem} 

Using this theorem, one gets the existence of a point at which the dominant eigenvalue is \(1\). One can use this to apply the following result (originally \cite[Theorem~2.33]{drmota2009random}) to find a local representation of our generating functions around the dominant singularity. The singularity is of the square-root type, giving us the generic behaviour of algebraic generating functions.

Let the number of equations in our system be \(N\). Though the general theorems and definitions that follow work for any finite set of \(q_i\)s, we only consider one \(q\)-variable. 

\begin{theorem}[Drmota \cite{drmota2009random}]
\label{thm:sqroot}
  Consider a non-linear system of functions \(\bm{\Phi}(z, \bm{Y}, q)\) that are analytic around \((0, \bm{0}, 0)\). Let \(\mathbb M\) be the Jacobian of \(\bm{\Phi}(z, \bm{y}, q)\) wrt \(\bm{y}\). Assume that:
  \begin{enumerate}
    \item All Taylor coefficients are non-negative.
    \item \(\bm{\Phi}\) is strongly connected.
  \item \(\bm{\Phi}(0, \bm{Y}, q) = \bm{0}\).
  \item \(\bm{\Phi}(z, \bm{0}, q) \neq \bm{0}\).
  \item \(\partial_z \bm{\Phi}(z, \bm{Y}, q) \neq \bm{0}\).
  \end{enumerate}
    Also, assume that the region of convergence of \(\bm{\Phi}\) is large enough such that the equations
  \begin{align}
    \bm{Y} &= \bm{\Phi}(z, \bm{Y}, 1), \\
    0 &= \det (I - \mathbb M(z, \bm{Y}, 1)),
  \end{align}
  have a common non-negative solution inside it. Let the solutions be \(z = z_c(q), \, \bm{Y} = \bm{Y_c}(q)\) in a complex neighbourhood \(U\) of \(q = 1\) that are real, positive, and minimal for positive real \(q \in U\).

  Let \(\bm{Y} = \bm{Y}(z, q) = \begin{bmatrix}
    Y_1(z, q) & Y_2(z, q) & \ldots & Y_N(z, q)
  \end{bmatrix}^{T}\) denote the analytic solutions of \(\bm{Y} = \bm{\Phi}(z, \bm{Y}, q)\) where \(\bm{Y}(0, q) = \bm{0}\).

  Then there exists an \(\epsilon > 0\) such that \(Y_j(z, q)\) admit a representation of the form
  \begin{align}
    Y_j(z, q) = g_j(z, q) - h_j(z, q) \sqrt{1 - \frac{z}{z_c(q)}}
  \end{align}
  for \(q \in U, \, |z - z_c(q)| < \epsilon, \, |\arg(z - z_c(q))| \neq 0\), where \(g_j, h_j\) are analytic functions with
  \begin{align}
    g_j(z, q) &\neq 0, \\
    h_j(z, q) &\neq 0, \\
    \begin{bmatrix}
      g_j(z_c(q), q)
    \end{bmatrix}_j &= \begin{bmatrix}
      Y_j(z_c(q), q)
    \end{bmatrix}_j = \bm{Y}_c(q).
  \end{align}
\end{theorem}

With this, one can apply a version of the Drmota-Lalley-Woods theorem (\cite[Theorem~1]{drmota1997}), which we state after the following definition. 

\begin{defn}
  We say our system \(\bm{\Phi}\) is of simple type if there exist \(2\)-D cones \(C_j \subseteq \mathbb R^{2}\) for \(1 \leq j \leq N\) such that they are centered at \(\bm{0}\) and have the following property:

For sufficiently large \((n, m) \in C_j\) we have \(y_{j, n, m} > 0\), where \(y_{j, n, m}\) are the Taylor coefficients of \(Y_j\),
\begin{align}
  Y_j(z, q) = \sum_{n, m} y_{j, n, m} z^n q^m \quad \text{for} \quad 1 \leq j \leq N.
\end{align}
\end{defn}
We note that a system being of simple type is a sufficient condition for there being only one dominant singularity that contributes to the asymptotics.

\begin{theorem}[Drmota \cite{drmota1997}]
\label{thm:expansion}
  Consider a non-linear system of functions \(\bm{\Phi}(z, \bm{Y}, q)\) that are analytic around \((0, \bm{0}, 0)\). Let \(\mathbb M\) be the Jacobian of \(\bm{\Phi}(z, \bm{y}, q)\) wrt \(\bm{y}\). Assume that:
  \begin{enumerate}
    \item All Taylor coefficients are non-negative.
    \item \(\bm{\Phi}\) is strongly connected.
  \item \(\bm{\Phi}(0, \bm{Y}, q) = \bm{0}\).
  \item \(\bm{\Phi}(z, \bm{0}, q) \neq \bm{0}\).
  \item \(\partial_z \bm{\Phi}(z, \bm{Y}, q) \neq \bm{0}\).
  \end{enumerate}
    Also, assume that the region of convergence of \(\bm{\Phi}\) is large enough such that the equations
  \begin{align}
    \bm{Y} &= \bm{\Phi}(z, \bm{Y}, 1), \\
    0 &= \det (I - \mathbb M(z, \bm{Y}, 1)),
  \end{align}
  have a common non-negative solution inside it. Let the solutions be \(z = z_c(q), \, \bm{Y} = \bm{Y_c}(q)\) in a complex neighbourhood \(U\) of \(q=1\). Define \(\lambda\) and \(\sigma^2\) as:
\begin{align}
  \lambda &= - \frac{1}{z_c(1)} \dfrac{\partial z_c}{\partial q_k}\bigg|_{q=1} && \sigma^2 = - \frac{1}{z_c(1)} \dfrac{\partial^2 z_c}{\partial q^2} \bigg|_{q=1}+ \lambda^2 + \lambda.
\end{align}
  If the system is of simple type and \(\sigma^2 > 0\), then the Taylor coefficients of 
  \begin{align}
    Y_j (z, q) = \sum_{n, m} y_{j, n, m} z^n q^m \quad \quad 1 \leq j \leq N
    \end{align}
  are asymptotically given, for \(m\) close to \(\lambda n\), by
  \begin{align}
    y_{j, n, m} = \frac{a_j (z_c(1))^{-n}}{2\sqrt{2}\pi n^2 }\left ( \exp \left ( -\frac{(m - \lambda n)^2}{2n\sigma^2} \right ) + \mathcal{O}(\sqrt{n})\right ),
  \end{align}
  uniformly for all \(n, m\), where \(y_{j, n, m} \neq 0\), and \(a_j \geq 0, \; 1 \leq j \leq N\).
\end{theorem}
In fact, \cite{drmota1997} and \cite{drmota2009random} also give implicit expressions for \(a_j\), but we don't need them in our case. 

We see that the asymptotics results require bounds on the variance of the statistic. For bivariate algebraic generating functions, one can sometimes find expressions for their moments. We will require the following lemma for that. 

\begin{lemma}
  \label{lem:limit}
  Consider a generating function \(f(z, q)\) which is locally, around \(z = r(q)\), of the form
  \begin{align}
    f(z, q) = a(z, q) + b(z, q)\sqrt{1 - \frac{z}{r(q)}}
  \end{align}
  where \(a(z, q)\) and \(b(z, q)\) each have radii of convergence \(> r(q)\). Let \(X\) denote the statistic counted by \(q\). Then:
\begin{align}
    \lim_{n \rightarrow \infty} \frac{\mathbb E [X]}{n} &= -\frac{r'(1)}{r(1)}, \\
    \lim_{n \rightarrow \infty} \frac{\mathbb V [X]}{n} &= -\frac{r''(1)}{r(1)} + \left (\frac{r'(1)}{r(1)}\right )^2 - \frac{r'(1)}{r(1)}.
  \end{align}
\end{lemma}
\begin{proof}
  Differentiating \(f\), we get
  \begin{align}
    \frac{\partial f}{\partial q} &= a' + b' \sqrt{1 - \frac{z}{r}} + \frac{1}{2} \frac{b \cdot r' \cdot z}{r^2} \left(1 - \frac{z}{r}\right)^{-1/2}, \\
    \frac{\partial^2 f}{\partial q^2} &= a'' + b'' \sqrt{1 - \frac{z}{r}} -z^2\left(1 - \frac{z}{r}\right)^{-3/2} \frac{b \cdot (r')^2}{4 r^4} \notag \\
    & \quad\quad + z \left (1 - \frac{z}{r} \right )^{-1/2} \left \{\frac{b' \cdot r'}{r^2} - \frac{b \cdot (r')^2}{r^3} + \frac{b \cdot r''}{2 r^2}\right \}.
  \end{align}
  Let \(R = r(1), R' = r'(1), R'' = r''(1)\). Then, on applying \cite[Theorem~VI.1]{Flajolet} (considering only the dominant term) and \cite[Proposition~III.2]{Flajolet} we get
  \begin{align}
    \mathbb E [X] &\sim - n \frac{R'}{R} + c_1, \\
    \mathbb E [X(X-1)] &\sim n^2 \left(\frac{R'}{R}\right )^2 + n \left \{\left (\frac{R'}{R}\right )^2 - \frac{2R' \cdot b'(R, 1)}{R \cdot b(R, 1)} - \frac{R''}{R}\right \} + c_2 
  \end{align}
  where \(c_1, c_2\) are \(\mathcal{O}(1)\) terms. These can be calculated explicitly by iterating the asymptotic method to get subdominant terms. Rearranging these, we get the required expressions. 
\end{proof}

\subsection{Partial asymptotics for \(\mathcal{G}(p_1,\dots,p_k)\)}
\label{sec:our_sys}

Here, we discuss the asymptotics of the system consisting of walks on the one-sided graphs. Recall the system from Section~\ref{sec:eqn_general}. The equations for walks on one-sided graphs are:
\begin{align}
  L_0^{(i)} &= 1 + z [L_1^{(i)} + qL_{p_i-1}^{(i)}], \\
  L_r^{(i)} &= z [L_{r-1}^{(i)} + L_{r+1}^{(i)}] \sum_{j \neq i} L_0^{(j)} && (1 \leq r \leq p_i - 2), \\
  L_{p_i - 1}^{(i)} &= z [L_{p_i-2}^{(i)} + q^{-1}L_{0}^{(i)}] \sum_{j \neq i} L_0^{(j)},
\end{align}
for all \(1 \leq i \leq k\).

In this first part of asymptotics, we consider the system consisting of only these equations. Equation~\eqref{eq:connect}, which relates these to walks on the entire graph, will only be considered Section~\ref{sec:asymp2} onwards.

In order to apply the asymptotics results, we see one minor and two major technical hurdles. First, the minor technicality; the winding number is not non-negative as required by the theorems. This can be easily fixed by the substitution \(z \mapsto zq\). Hence, we pick up an additional \(q\) with every step, which is equivalent to the invertible transformation \(W \mapsto W + n\). This shifts the mean of the statistic to \(n\) and makes it non-negative while keeping the variance the same. Hence, we can apply this transformation, use the theorems, and then shift the mean back to \(0\). 

The two major technical hurdles we face are:
\begin{itemize}
  \item Bounding the variance of winding number.
  \item Proving the system is of simple type.
\end{itemize}
We deal with the former in Section~\ref{sec:pat_dec_var}, and with the latter in Section~\ref{sec:transform}.

Using those, one can get the following local limit law for our system of \(L^{(i)}_j\), specialised to even lengths and even winding numbers. We note here that for this particular law, we do not consider the \(P_i\)s and \(F_k\) in the system as discussed above. 

\begin{theorem}
  \label{thm:lll}
  For any finite system of facets, let \(L_0^{(i)} = \sum_{n, m} c_{i, n, m} z^n q^m\). Then, asymptotically
  \begin{align}
    c_{i, 2n, 2m} = \frac{a_i \beta^{-2n}}{n^2} \left (\exp\left (\frac{-m^2}{\gamma n} \right ) + \mathcal{O}(\sqrt{n}) \right),
  \end{align}
  when \(m\) is close to \(0\), for some positive constants \(a_i, \beta, \gamma\).
\end{theorem}
The proof is described in Section~\ref{sec:transform}, and follows by combining Lemma~\ref{lem:even}, Lemma~\ref{lem:odd}, and Lemma~\ref{lem:mixed}.

This gives us the asymptotics of the system with walks on the one-sided graphs. We will use this to find a local limit law for \(F_k\) in Section~\ref{sec:asymp2}.

\section{Decompositions, patterns, and variance}
\label{sec:pat_dec_var}

The first substantial issue we face for asymptotics is that the variance should be bounded on both sides. Hence, we show here that the variance of the winding number is \(\Theta(n)\). 

The tools for proving the upper bound are built in Section~\ref{sec:decomp}, and for the lower bound in Section~\ref{sec:pat}. Then Section~\ref{sec:var} puts them together for our bound.

\subsection{Decomposing the walk}
\label{sec:decomp}
In this section, we see how to decompose our walk into constituent subwalks, each with its own winding number contribution. This allows us to represent the generating function \(F_k(z; q)\) in a different way, using which we can find a linear upper bound on the variance (Theorem~\ref{thm:var}).

\begin{lemma}[Decomposition]
\label{lem:decomp}
  \(F_k(z; q) = \sum_{m, n} f_{m, n} q^m z^n\). We can represent it in the form
\begin{align}
  F_k(z; q) = \sum_n \sum_{\ell = 0}^n d_{n, \ell}z^n \left (q + \frac{1}{q} \right)^{\ell},
\end{align}
  where \(d_{n, \ell} \geq 0\) is the number of walks of length \(n\) with \(\ell\) non-trivial loops.
\end{lemma}

\begin{proof}
  In any closed walk on our graph, we get a non-zero contribution for the winding number any time all the following things happen:
\begin{enumerate}
  \item The walk enters a new facet, going farther away from the root.
  \item The walk subsequently stays in the facet, or goes farther away from the root.
  \item The walk returns to the facet, and moves back closer to the root from a \textit{different} edge than it entered the facet with.
\end{enumerate}

\begin{figure}[h]
  \centering
  \includegraphics[width=0.7\linewidth]{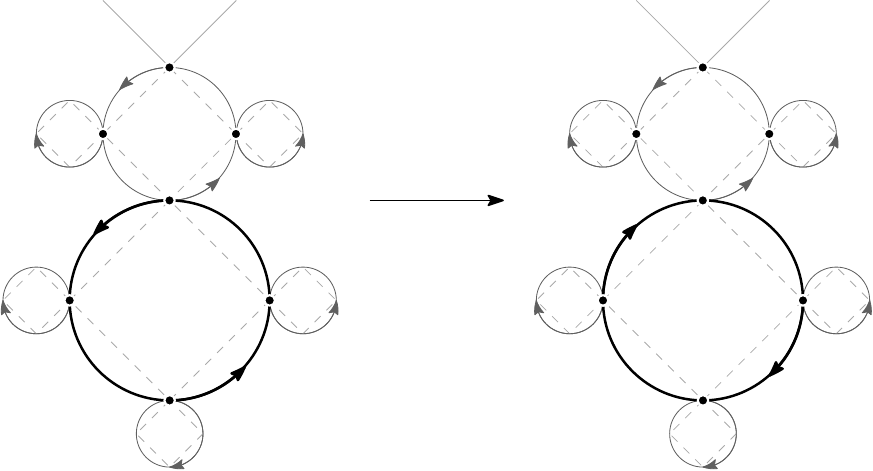}
  \caption{Walk decomposition and switching the winding number contribution of a particular loop. Notice that one can reverse the direction of the walk on the main facet without changing the orientation on any of the other facets.}
  \label{fig:decomp}
\end{figure}
Hence, every facet can independently contribute a \(0, 1\), or \(-1\) to the winding number. We note that if the walk enters a facet multiple times, each instance is treated independently. When a walk leaves a facet, it must re-enter the facet using exactly the same vertex due to the Schreier graph having a \textit{tree-of-polygons} structure. This happens since every vertex is a cut vertex in our graph.

So, in a closed walk, all the non-zero contributions to the winding number can be independently switched in sign by reversing the walk in the corresponding facet. Therefore, if we have \(\ell\) non-trivial contributions to the winding number in a walk, we can switch any subset of them around independently. Therefore \(q\) and \(q^{-1}\) are equiprobable, and the result follows.
\end{proof}

This tells us that every individual contributing loop can be flipped. Hence, we do not have very heavy tails since the contributing loops can be flipped in the opposite direction to get a value closer to the mean.

\subsection{Pattern theorem}
\label{sec:pat}
In this section, we prove a pattern theorem for closed walks on our graphs. In general, a pattern theorem shows that a small \textit{pattern} or a \textit{motif} occurring in a combinatorial class implies that most large objects in that class contain a positive density of that pattern. Pattern theorems are ubiquitous in combinatorics literature, and the first pattern theorem is attributed to Kesten \cite{pattern} in the setting of self-avoiding walks.

In particular, our pattern would be a cycle across the facet of the largest length in our configuration. Our case is somewhat simpler due to the vertex-transitivity of our Schreier graphs. However, even without that, one can make this argument along the lines of \cite[Section~3]{chapman17} and \cite{bendergaorich}. This leads to the lower bound on variance in Theorem~\ref{thm:var}.

Fix a given group \(\mathcal{G}_k\), and let the largest \(p_i\) be \(p\). We assume that \(p > 2\), since we have explicitly computed the case that all \(p_i=2\) in Section~\ref{sec:special_case}. We fix our pattern to be a directed closed cycle on \(p\) vertices.

\begin{defn}
  Let \(\mathcal{M}\) be the set of closed walks on the Schreier graph of \(\mathcal{G}_k\). Also, let \(M(z) = \sum_{y \in \mathcal{M}} z^{|y|} = \sum_n m_n z^n\).
\end{defn}
Clearly, the pattern naturally occurs in these walks. We now try to show that it must occur a linear number of times, in \textit{most} of the walks.

\begin{defn}
  Let \(c>0\) be fixed and small. Let \(\mathcal{H}\) be the class of all walks \(y\) with at most \( c|y|\) patterns. This is the class of walks with density of patterns at most \( c\). Also let \(H(z)\) be the corresponding generating function.
\end{defn}

The main technical step of this section is to show that the exponential growth rate of \(\mathcal{H}\) is strictly smaller than that of \(\mathcal{M}\). We do this by showing that the radius of convergence of the generating function \(H(z + z^p)\) is strictly larger than that of \(M(z)\). To do this, we need to have a solid combinatorial interpretation of the substitution \(z \mapsto z +z^p\); we should understand this as a process in which at each vertex, independently, we either leave it alone (\(z \mapsto z\)) or insert a length \(p\) walk (\(z \mapsto z^p\)). There are some technical conditions require to ensure that this substitution is consistent, and we summarise those as an attachment scheme (see \cite{bendergaorich} for other examples of this idea).

\begin{defn}
  An attachment scheme is a procedure to attach patterns to elements in \(\mathcal{H}\) that has the following properties:

\begin{enumerate}
  \item The pattern addition preserves the vertex where we break the walk to add it.
  
  \item The pattern attachment is \textit{not} an iterative process. Once a pattern is attached, any new places it creates for viable attachments can \textit{not} be used. In this way, the pattern addition can be \textit{parallelised}, and so nested patterns cannot be added. We might have nested patterns when the base walk already has some patterns, but no nesting is permitted in the newly added patterns (this is because we are making a single substitution).
  
  \item We can add a pattern at every vertex of the walk.
  
  \item All patterns and walks lie in \(\mathcal{M}\) that is closed under any number of attachments of the pattern. Hence, the result of this substitution process must stay within \(\mathcal{M}\).
\end{enumerate}
\end{defn}

We note here that we count \textit{labelled} vertices here, with the starting vertex being numbered \(0\). So a walk of length \(4\) that just bounces between \(2\) vertices would still be counted as being on \(4\) vertices. (This convention gives us the convenience that our generating function for walks on \(n\) vertices is the same as the generating function for walks of length \(n\)).

We define our addition scheme as follows. Consider a walk of length \(n\). At each of those vertices, either do nothing, or break the walk, attach a \(p\)-cycle (our pattern), and continue the walk again from the same vertex.

We note that due to the structure of the Schreier graph, every vertex can potentially have a \(p\)-cycle attached to it. Hence, this is a valid attachment process.

\begin{figure}[h]
  \centering
  \includegraphics[width=0.8\linewidth]{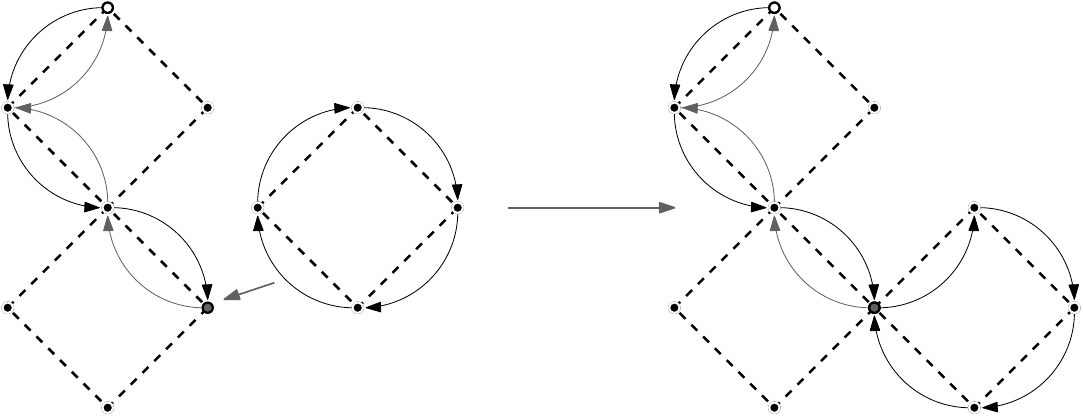}
  \caption{Example of attachment with our walks.}
  \label{fig:pattern}
\end{figure}

\begin{defn}
  Let \(\mathcal{K}\) be the (multi-)set of walks obtained by using the attachment procedure on \(\mathcal{H}\). While some caution is needed to deal with the multiplicities, we can (thankfully) manage to bound the multiplicities to get our pattern theorem.
\end{defn}

Then, one can write this attachment as a generating function substitution.
\begin{align}
  H(z) &= \sum_{y \in \mathcal{H}} z^{|y|} = \sum_n h_n z^n, \\
  K(z) &= \sum_{y \in \mathcal{K}} z^{|y|} = \sum_n k_n z^n &&\text{(Has multiplicity)}\\
 &= H(z + z^p). &&\text{(Attachment process)}
\end{align}
From this, one gets an expression for the growth rate of \(K\). This is the restatement of \cite[Lemma~2]{bendergaorich} with our parameters.
\begin{lemma}[Bender et al. \cite{bendergaorich}]
\label{lem:radius}
  For \(H(z)\) and \(K(z)\) described above, we get that
  \begin{align}
    \frac{1}{\mu(H)} = \frac{1}{\mu(K)} + \frac{1}{\mu(K)^p}.
  \end{align}
\end{lemma}
\begin{theorem}[Pattern Theorem]
  \label{thm:pattern}
  Under the setting described above, where we have a valid attachment procedure, we get \(\mu(M) > \mu(H)\). Consequently, the set of walks without a density of patterns is exponentially small in the set of all walks.

\end{theorem}
We remind the reader that the definition of the set \(\mathcal{H}\) implicitly depends on a constant \(c > 0\). So the above statement implicitly asserts the existence of some \(c > 0\) that defines an \(\mathcal{H}\) that makes the above true; see Corollary~\ref{cor:pattern} that follows the proof. The proof below follows \cite[Section~3]{chapman17}.
\begin{proof}
  Consider \(y \in \mathcal{M}\) with length \(n\). So, \(y\) is either in \(\mathcal{K}\) or not. Let \(s_n\) count the number of walks on \(n\) vertices not in \(\mathcal{K}\). Let \(k_{n, \ell}\) be the number of walks in \(\mathcal{K}\) with length \(n\) and having \(\ell\) patterns.
  
  For any \(x \in \mathcal{K}\) with length \(n\), let \(\#_n(x)\) be its multiplicity. Also, let \(\Tilde{x_n} \in \mathcal{G}\) be the walk with the highest multiplicity with length \(n\). So
\begin{align}
  m_n = s_n + \sum_{\substack{y \in \mathcal{K} \\ |y| = n}} \frac{1}{\#_n(y)} \geq \sum_{\substack{y \in \mathcal{K} \\ |y| = n}} \frac{1}{\#_n(y)} \geq \sum_{\substack{y \in \mathcal{K} \\ |y| = n}} \frac{1}{\#_n(\Tilde{x_n})} = \frac{k_n}{\#_n(\Tilde{x_n})}.
\end{align}
Since \(\#_n(\Tilde{x_n})\) is the maximum multiplicity of any walk, we have
\begin{align}
  \#_n(\Tilde{x_n}) &= \max \left \{\#_n(y) : y \in \mathcal{K} \right \} \\
  &= \max \left \{ \sum_{j \geq \ell - cn} \binom{\ell}{j} : \ell \leq n, \;k_{n, \ell} \neq 0\right \} & \text{(each pattern is replaced or not)}\\
  &= \max \left \{ \sum_{i < cn} \binom{\ell}{i} : \ell \leq n, \;k_{n, \ell} \neq 0 \right \} & \text{(symmetry of the binomial)}\\ 
  &\leq \max \left \{ \sum_{i < cn} \binom{n}{i} : \ell \leq n, \;k_{n, \ell} \neq 0 \right \} & \text{(since } \ell \leq n \text{)} \\
  &= \sum_{i < cn} \binom{n}{i} \leq cn\binom{n}{cn} \leq cn \left ( \frac{e}{c} \right )^{cn} = t_n.
\end{align}
Hence, we get \(m_n \geq k_n / t_n\), where \(t_n\) implicitly depends on \(c\). Now,
\begin{align}
  \mu(M) &= \limsup_n (m_n^{1/n}) \\
  &\geq \limsup_n \left ( \frac{k_n}{t_n} \right )^{1/n} = \mu(K) \left (\frac{c}{e} \right )^c \lim_n \left ( \frac{1}{cn}\right )^{1/n} = \mu(K) \left (\frac{c}{e} \right )^c.
\end{align}
Now, we can use Lemma~\ref{lem:radius} on the above expression.
\begin{align}
  \frac{\mu(M)}{\mu(H)} \geq \left (\frac{c}{e} \right )^c \frac{\mu(K)}{\mu(H)} &= \left (\frac{c}{e} \right )^c \mu(K) \left ( \frac{1}{\mu(K)} + \frac{1}{\mu(K)^p}\right ) \\
  & = \left (\frac{c}{e} \right )^c \left ( 1 + \mu(K)^{-p-1}\right ) > 1.
\end{align}
The last line follows since \((c/e)^c \longrightarrow 1\). Additionally, notice that \(\mu(H)\) is bounded below by the growth rate of returns on trees, and bounded above by twice the number of generators of the group. Hence \(0 < \mu(H) \leq 2k < \infty\). Then, Lemma~\ref{lem:radius} implies \(0 < \mu(K) < \infty\).
\end{proof}

\begin{cor}
  For our case of closed walks on graphs, there exist constants \(c, d > 0\) such that for a random walk of length \(n\), say \(y \in \mathcal{M}\),\label{cor:pattern}
  \begin{align}
    \mathbb P[y \text{ has } \leq cn \text{ patterns}] \leq e^{-dn}.
  \end{align}
\end{cor}

\subsection{Variance}
\label{sec:var}

Now, we can use Corollary~\ref{cor:pattern} and Lemma~\ref{lem:decomp} to prove that the variance of the winding number of a closed random walk of length \(n\) is \(\Theta(n)\).

We will need this technical but standard probability lemma. There are many proofs of this available. For one using moment generating functions, see \cite[Example~1.5(A)]{ross}.
\begin{lemma}
\label{lem:prob}
Let \(\{X_i\}_i\) be a sequence of iid RVs with mean \(\lambda\) and variance \(\sigma^2\). Let \(S_n = \sum_{i=0}^n X_i\) be the sum of the first \(n\) RVs. Let \(N\) be a (bounded) random variable taking values in \(\{0, 1, 2, \ldots \}\) that is independent of the \(X_i\)s. Then
  \begin{align}
    \mathbb E[S_N] &= \lambda \mathbb E [N], \\
    \mathbb V [S_N] &= \lambda^2 \mathbb V [N] + \sigma^2 \mathbb E [N].
  \end{align}
\end{lemma}

The lower bound comes from the pattern theorem we proved. Heuristically, since we have a lot of small patterns, we can flip each of those independently to give the walk different values of winding number. Hence, the variance cannot be too small. The upper bound directly follows from the decomposition of the walk in Lemma~\ref{lem:decomp}.
\begin{theorem}
 \label{thm:var}
  Let \(W\) be the winding number of a closed random walk. Then \(\mathbb V [W] = \Theta(n)\).
\end{theorem}
\begin{proof}
  For the lower bound, consider any walk of length \(n\). Here, we can separate the winding number contribution from the patterns, and from the ``base'' walk after removing the patterns. This idea is inspired by \cite{selfavoiding}.

  Let \(W\) be the winding number of a random walk, and \(W_b, W_p\) denote the contributions from the base walk and patterns, respectively. Then \(W = W_b + W_p\). Also we note \(W_b\) and \(W_p\) are independent. So, we have \(\mathbb V [W] = \mathbb V [W_b] + \mathbb V [W_p]\). 

  To proceed further, let the winding number contribution of the \(i\)-th pattern be given by \(X_i\). Let \(N\) be the number of patterns. For a walk of length \(n\), the random variable \(N\) is bounded above by \(n\). Then, by the notation in Lemma~\ref{lem:prob}, \(W_p \equiv S_N\). 

We also note that \(X_i\)s are iid Rademacher, that is, they are discrete random variables with distribution \(\mathbb P[X_i = 1] = \mathbb P[X_i = -1] = 1/2\). Hence:
\begin{align}
  \mathbb V [W] &= \mathbb V [W_b] + \mathbb V [W_p] \\
  &= \mathbb V [W_b] + \mathbb V [S_N] \\
  &= \mathbb V [W_b] + \mathbb V [X_i] \mathbb E [N] && \text{(Since \(\mathbb E [X_i] = 0\))} \\
  &\geq \mathbb E [N] && \text{(Since \(\mathbb V [X_i] = 1\))} \\
  &\geq cn(1 - e^{-dn}) && \text{(From Corollary~\ref{cor:pattern})}\\
  &= \Omega(n).
\end{align}
For the upper bound, consider the expression of \(F_k\) in Lemma~\ref{lem:decomp}:
\begin{align}
  F_k(z; q) = \sum_n \sum_{\ell = 0}^n d_{n, \ell}z^n \left (q + \frac{1}{q} \right)^{\ell}.
\end{align}
Now, one can use \cite[Proposition~III.2]{Flajolet} to get
\begin{align}
  \mathbb V [W] &= \frac{\sum_{\ell = 0}^n d_{n, \ell} 2^{\ell} \ell}{\sum_{\ell = 0}^n d_{n, \ell} 2^{\ell}} \leq \frac{n\sum_{\ell = 0}^n d_{n, \ell} 2^{\ell}}{\sum_{\ell = 0}^n d_{n, \ell} 2^{\ell}} = n.
\end{align}
Hence, putting these together, \(\mathbb V [W] = \Theta(n)\).
\end{proof}

We also note that the variance via patterns and decomposition here has used the final generating function \(F_k\). However, the same arguments work for the generating functions \(L^{(i)}_0\).

\section{Transformations and proof of Theorem~\ref{thm:lll}}
\label{sec:transform}

The second substantial issue we face is that our system is not of simple type due to parity effects in the coefficients, which stem from the parities of the facets. To deal with this, we consider three different cases:
\begin{enumerate}
  \item all facets even [even case]
  \item all facets odd [odd case]
  \item at least one odd and one even facet [mixed case]. 
\end{enumerate}

As a result, we cannot immediately apply the results of Section~\ref{sec:asymp1} to our system of equations. Thankfully, there are some simple transformations of our system that make it amenable to these results, which we describe in this section.

\subsection{Even case}
Here, all \(p_i\) are even, and hence all the facets have even length. We note that the coefficients of odd powers of \(z\) in \(L_0^{(j)}\) are all zero. So, the system has a symmetry \(L_0^{(i)}(z; q) = L_0^{(i)}(-z; q)\). Hence, near \(q=1\), we get two singularities \(z_c(1)\) and \(-z_c(1)\) with equal and opposite contributions. To deal with this, we restrict our system to give us the asymptotics along the even-length walks. This removes the additional singularity.

We describe the transformations required by an example which can be generalised in a straightforward manner. Consider the case with a square and a hexagon (\(p_1=4, p_2=6\)):

\begin{subequations}
\begin{minipage}{0.49\textwidth}
  \begin{align}
  L^{(1)}_0 &= 1 + z[L^{(1)}_1 + qL^{(1)}_3], \\
  L^{(1)}_1 &= z[L^{(1)}_0 + L^{(1)}_2]L^{(2)}_0, \\
  L^{(1)}_2 &= z[L^{(1)}_1 + L^{(1)}_3]L^{(2)}_0, \\
  L^{(1)}_3 &= z[L^{(1)}_2 + q^{-1}L^{(1)}_0]L^{(2)}_0, \\
  L^{(2)}_0 &= 1 + z[L^{(2)}_1 + qL^{(2)}_5],
  \end{align}
\end{minipage}
\begin{minipage}{0.49\textwidth}
  \begin{align}
  L^{(2)}_1 &= z[L^{(2)}_0 + L^{(2)}_2]L^{(1)}_0, \\
  L^{(2)}_2 &= z[L^{(2)}_1 + L^{(2)}_3]L^{(1)}_0, \\
  L^{(2)}_3 &= z[L^{(2)}_2 + L^{(2)}_4]L^{(1)}_0, \\
  L^{(2)}_4 &= z[L^{(2)}_3 + L^{(2)}_5]L^{(1)}_0, \\
  L^{(2)}_5 &= z[L^{(2)}_4 + q^{-1}L^{(2)}_0]L^{(1)}_0.
  \end{align}
\end{minipage}
\end{subequations}

We note that the expressions arising in the even-length walks can further be written in terms of even-length walks:
\begin{subequations}
\begin{align}
  L_1^{(1)} + qL_3^{(1)} &= z [2L_0^{(1)} + L_2^{(1)} + qL_2^{(1)}]L_0^{(2)}, \\
  L_1^{(1)} + L_3^{(1)} &= z[L_0^{(1)} + q^{-1}L_0^{(1)} + 2L_2^{(1)}]L_0^{(2)}, \\
  L_1^{(2)} + qL_5^{(2)} &= z [2L_0^{(2)} + L_2^{(2)} + qL_2^{(4)}]L_0^{(1)}, \\
  L_1^{(2)} + L_3^{(2)} &= z[L_0^{(2)} + L_4^{(2)} + 2L_2^{(2)}]L_0^{(1)}, \\
  L_3^{(2)} + L_5^{(2)} &= z[L_2^{(2)} + q^{-1}L_0^{(2)} + 2L_4^{(2)}]L_0^{(1)},
\end{align}
\end{subequations}
Substituting these back into the system, we get
\begin{subequations}
  \begin{align}
  L^{(1)}_0 &= 1 + z^2 L^{(2)}_0 [2L^{(1)}_0 + L^{(1)}_2 + qL^{(1)}_2], \\
  L^{(1)}_2 &= z^2 (L^{(2)}_0 )^2[L^{(1)}_0 + q^{-1}L^{(1)}_0 + 2L^{(1)}_2], \\
  \notag{} \\
  L^{(2)}_0 &= 1 + z^2 L^{(1)}_0 [2L^{(2)}_0 + L^{(2)}_2 + qL^{(2)}_4], \\
  L^{(2)}_2 &= z^2 (L^{(1)}_0 )^2[L^{(2)}_0 + 2L^{(2)}_2 + L^{(2)}_4], \\
  L^{(2)}_4 &= z^2 (L^{(1)}_0 )^2[L^{(2)}_2 + 2L^{(2)}_4 + q^{-1}L^{(2)}_0].
\end{align}
\end{subequations}
Now, we can substitute \(z^2 = t\) in this new set of equations and get the system we need. This tells us that all even-length walks have only one singularity, so we can apply the results above.

\begin{lemma}
  In the setting above with all even lengths, the system is of simple type. 
\end{lemma}
\begin{proof}
  We show this for an arbitrary \(L^{(i)}_0\). Similar arguments work for other \(L^{(i)}_j\) generating functions. With some abuse of notation, let the \(L^{(i)}_0\) post-transformation be \(L^{(i)}_0(t; q) = \sum_{n, m} d_{i, n, m} t^n q^m\). The transformation gives us the mapping \(d_{i, n, m} = c_{i, 2n, m}\). Indeed, one can think of \(n\) being the ``half-length'' of a walk.

  First, note that \(d_{i,0,0} = 1\), and then by stepping out and back along any incident edge, we have \(d_{i,n+1,0} \geq d_{i,n,0}\). Then, by walking around the smallest facet in either direction, we have that \(d_{i,n+p/2,m\pm1} \geq d_{i,n,m}\), where \(p\) is the length of the shortest facet.

  By iterating these relations, we have that \(d_{i,n,m} > 0\) for any \( |m| \leq \lfloor 2n/p \rfloor \), giving us the required cone of positive coefficients.
\end{proof}
\begin{lemma}
\label{lem:even}
  For the setting above with all even lengths, the non-zero Taylor coefficients of \(L^{(i)}_0(t; q) = \sum_{n, m}d_{i, n, m}t^n q^m\) (when \(m\) is close to \(0\)) are asymptotically given by
  \begin{align}
    d_{i, n, m} = \frac{a_i \beta^{-n}}{2^{3/2} \pi n^2} \left (\exp\left (\frac{-m^2}{2 \gamma n} \right ) + \mathcal{O}(\sqrt{n}) \right),
  \end{align}
  for positive constants \(a_i, \beta, \gamma\).
\end{lemma}
\begin{proof} We begin by substituting \(\Tilde{L}^{(i)} = L^{(i)}_0 - 1\) in our system. The new system now is in the variable \(\bm{Y} = \begin{bmatrix}
  \Tilde{L}^{(1)} & L^{(1)}_1 & \cdots & \Tilde{L}^{(2)} & L^{(2)}_1 & \cdots & \Tilde{L}^{(k)} & L^{(k)}_1 & \cdots
\end{bmatrix}^T\). Let the equations by denoted by \(\bm{\Phi}(t, \bm{Y}, q)\). This is the system we discussed earlier with all the even-length walks.

We can check that this system does indeed satisfy the conditions in Theorem~\ref{thm:eigenvalue}. Hence, there is a solution to \(\bm{Y}(t, 1) = \bm{\Phi}(t, \bm{Y}, 1)\) where \(\mathbb M(t,\bm{Y}, 1)\) has dominant eigenvalue \(1\). This gives us that
\begin{align}
  \bm{Y_c}(t_c(1), 1) = \bm{\Phi}(t_c(1), \bm{Y_c}(1), 1), \\
   \det(I - \mathbb M(t,\bm{Y}, 1)) = 0
\end{align}
are satisfied. We define \(\lambda\) and \(\sigma^2\) by
\begin{align}
  \lambda &= - \frac{1}{t_c(1)} \frac{\partial t_c}{\partial q} \bigg|_{q=1} && \sigma^2 = - \frac{1}{t_c(1)} \frac{\partial^2 t_c}{\partial q^2}\bigg|_{q=1}+ \lambda^2 +\lambda.
\end{align}
We can apply Theorem~\ref{thm:sqroot} here to get a local representation of \(Y_j(t, q)\) of the form
\begin{align}
  Y_j(t, q) = g_j(t, q) - h_j(t, q) \sqrt{1 - \frac{t}{t_c(q)}},
\end{align}
where \(g_j, h_j\) are analytic functions. 

Define \(X\) to be the random variable that counts the winding number. Now, from Lemma~\ref{lem:limit}, we see that
\begin{align}
  \lim_{n \rightarrow \infty} \frac{\mathbb E[X]}{n} = \lambda \quad \text{and} \quad \lim_{n \rightarrow \infty} \frac{\mathbb V[X]}{n} = \sigma^2.
\end{align}
Now, from our earlier discussion on shifting the mean, we can take \(\lambda = 0\). Then, using Theorem~\ref{thm:var} and Lemma~\ref{lem:limit} we get that \(\sigma^2 = \gamma\) for some constant \(\gamma > 0\). In fact, for explicit cases, we can calculate it using the same procedure.

Now, applying Theorem~\ref{thm:expansion}, we get that asymptotically
  \begin{align}
    d_{i, n, m} = \frac{a_i (t_c(1))^{-n}}{2^{3/2} \pi n^2} \left (\exp\left (\frac{-m^2}{2 \gamma n} \right ) + \mathcal{O}(\sqrt{n}) \right),
  \end{align}
when \(m\) is close to \(0\), and for some positive constants \(a_i\) (depending on the solution to the system). We only look at the expansion of \(\Tilde{L}^{(i)}\) here, but a similar expansion holds for the entire system.
\end{proof}

\subsection{Odd case}
Here, we note that nonzero coefficients can only occur when powers of \(q\) and \(z\) have the same parity. Hence, our system has the symmetry \(L_0^{(i)}(z; q) = L_0^{(i)}(-z; -q)\). Since a local limit law depends on the behaviour of the function on the entire unit circle (\(|q| = 1\)), this gives us contributions from both \(z_c(1)\) and \(-z_c(-1)\).

Here, we can remove the parity effects by substituting \(q^2 \mapsto u\). In order to keep the system polynomial, we also substitute \(z \mapsto t q\). This is the same substitution as shifting the mean.

As an example, we consider the case with a triangle and a pentagon (\(p_1=3, p_2=5\)):

\begin{subequations}
\begin{minipage}{0.49\textwidth}
  \begin{align}
  L^{(1)}_0 &= 1 + t \sqrt{u}[L^{(1)}_1 + \sqrt{u}L^{(1)}_2], \\
  L^{(1)}_1 &= t \sqrt{u}[L^{(1)}_0 + L^{(1)}_2]L^{(2)}_0, \\
  L^{(1)}_2 &= t \sqrt{u}[L^{(1)}_1 + \sqrt{u^{-1}}L^{(1)}_0]L^{(2)}_0, \\
  L^{(2)}_0 &= 1 + t \sqrt{u}[L^{(2)}_1 + \sqrt{u}L^{(2)}_4],
  \end{align}
\end{minipage}
\begin{minipage}{0.49\textwidth}
  \begin{align}
  L^{(2)}_1 &= t \sqrt{u}[L^{(2)}_0 + L^{(2)}_2]L^{(1)}_0, \\
  L^{(2)}_2 &= t \sqrt{u}[L^{(2)}_1 + L^{(2)}_3]L^{(1)}_0, \\
  L^{(2)}_3 &= t \sqrt{u}[L^{(2)}_2 + L^{(2)}_4]L^{(1)}_0, \\
  L^{(2)}_4 &= t \sqrt{u}[L^{(2)}_3 + \sqrt{u^{-1}}L^{(2)}_0]L^{(1)}_0.
  \end{align}
\end{minipage}
\end{subequations}

Substituting the odd-length walks (\(L_1^{(1)}, L_1^{(2)}, L_3^{(2)}\)) back into the system we get
\begin{subequations}
\begin{align}
  L_0^{(1)} &= 1 + t [uL_2^{(1)} + tu(L_0^{(1)} + L_2^{(1)})L_0^{(2)}], \\
  L_2^{(1)} &= t L_0^{(2)} [t u L_0^{(2)} (L_0^{(1)} + L_2^{(1)}) + L_0^{(1)}],\\
  \notag \\
  L_0^{(2)} &= 1 + t u [L_4^{(2)} + t(L_0^{(2)} + L_2^{(2)})L_0^{(1)}], \\
  L_2^{(2)} &= t^2 u (L_0^{(2)})^2 [L_0^{(2)}+ 2L_2^{(2)}+ L_4^{(2)}], \\
  L_4^{(2)} &= t L_0^{(1)} [tu L_0^{(1)} (L_4^{(2)} + L_2^{(1)}) + L_0^{(2)}].
\end{align}
\end{subequations}
This gives us the system we can use. 

\begin{lemma}
  In the setting above with all odd lengths, the system is of simple type. 
\end{lemma}
\begin{proof}

  We show this for an arbitrary \(L^{(i)}_0\). Similar arguments work for the other \(L^{(i)}_j\) generating functions. With some abuse of notation, let the generating functions \(L^{(i)}_0\) post-transformation be \(L^{(i)}_0(t; u) = \sum_{n, m} d_{i, n, m} t^n u^m\). 
  
  The sequence of transformations is \((z, q) \mapsto (z, \sqrt{u}) \mapsto (t\sqrt{u}, \sqrt{u})\). Hence, the transformations give us the relations 
  \begin{align}
    d_{i, n, m} = c_{i, n, 2m-n} && c_{i, n, m} = d_{i, n, \frac{m+n}{2}}.
  \end{align}
  This works because the only nonzero coefficients initially were the ones with \(m\) and \(n\) having the same parity.
  
  First, note that \(d_{i, n, n/2} = c_{i, n, 0} > 0\) for even \(n\) by stepping out and back along any incident edge. Also, for sufficiently large odd \(n\), \(d_{i, n, (n \pm 1)/2} = c_{i, n, n \pm 1} > 0\).
  
  Then, by walking around the smallest facet in either direction, we have that \(d_{i, n+p, (n+p+m \pm 1)/2} = c_{i,n+p,m\pm1} \geq c_{i,n,m} = d_{i, n, (n+m)/2}\), where \(p\) is the length of the shortest facet, and \(p\) odd.

  By iterating these relations, we get the required cone with axis \(m = n/2\).
\end{proof}

\begin{lemma}
\label{lem:odd}
  For the setting above with all odd lengths, the non-zero Taylor coefficients of \(L^{(i)}_0(t; u) = \sum_{n, m}d_{i, n, m}t^n u^m\) (when \(m\) is close to \(n\)) are asymptotically given by
  \begin{align}
    d_{i, n, m} = \frac{a_i \beta^{-n}}{2^{3/2} \pi n^2} \left (\exp\left (\frac{-(m-n)^2}{2 \gamma n} \right ) + \mathcal{O}(\sqrt{n}) \right),
  \end{align}
  for positive constants \(a_i, \beta, \gamma\).
\end{lemma}
The proof follows from the same argument as in Lemma~\ref{lem:even}.

\subsection{Mixed case}
In case we have both even and odd length facets, neither of the issues occurs. We can have loops of both even and odd length contributing to the winding number. The system is actually of simple type. Hence, we do not have any unwanted symmetries in the system and require no transformations. So, we only have one singularity and saddle point corresponding to \(z_c(1)\).

\begin{lemma}
  In the setting with at least one even and one odd length, the system is of simple type. 
\end{lemma}
\begin{proof}
  We show this for an arbitrary \(L^{(i)}_0\). Similar arguments work for other \(L^{(i)}_j\) generating functions.  

  First, note that \(c_{i, n, 0} > 0\) for all even \(n\) by stepping out and back along any edge. Now, let the smallest even \(p_i\) in the group presentation of \(\mathcal{G}_k\) be \(p_1\), and the smallest odd \(p_i\) be \(p_2\). 
  
  Now, by walking along the smallest odd or even facet in either direction, we get \(c_{i, n + p_1, m \pm 1} \geq c_{i, n, m}\) and \(c_{i, n + p_2, m \pm 1} \geq c_{i, n, m}\). Since \(p_1\) and \(p_2\) have different parities, iterating these relations gives us the required cone.
\end{proof}

\begin{lemma}
\label{lem:mixed}
  For a mixed system, the non-zero Taylor coefficients of \(L^{(i)}_0(z; q) = \sum_{n, m}c_{i, n, m}z^n q^m\) (when \(m\) is close to \(0\)) are asymptotically given by
  \begin{align}
    c_{i, n, m} = \frac{a_i \beta^{-n}}{2^{3/2} \pi n^2} \left (\exp\left (\frac{-m^2}{2 \gamma n} \right ) + \mathcal{O}(\sqrt{n}) \right),
  \end{align}
  for positive constants \(a_i, \beta, \gamma\).
\end{lemma}
The proof, again, follows from the same argument as in Lemma~\ref{lem:even}.

\section{Part 2: Asymptotics of \(F_k\)}
\label{sec:asymp2}

We now move the asymptotics of \(L_0\) proven in Section~\ref{sec:asymp1} to the asymptotics of \(F_k\). This is done by showing that a walk on the entire Schreier graph spends very little time at the root, and hence the asymptotics are governed by the asymptotics of the walks in the one-sided graphs. This will prove the main result of the paper.

We first note that Theorem~\ref{thm:eigenvalue} implies that all the \(L_j^{(i)}\) in our system have the same growth rate. Additionally,
\begin{align}
  L_0^{(1)} = \frac{1}{1-P_1} && F_k = \frac{1}{1 - \sum_{j=1}^k P_k}.
\end{align}

\begin{defn}
  We choose any arbitrary \(L_0^{(i)}(z; q)\) (say \(L_0^{(1)}(z; q)\)) and call it \(L(z; q)\). We will use this for our asymptotics.
\end{defn}

This is okay to do since \(L_j^{(i)}\) all have the same growth rate, and all \(L_0^{(i)}\) have the same asymptotics. 

The asymptotics of \(L\) are not those of a simple pole, and hence the dominant singularity does not arise from the zero of the denominator. Instead, it must come from \(P\), and so \(L\) and \(P_1\) have the same radius of convergence. Since all the \(L_j\) have the same radius of convergence, we know that all the \(P_j\) have the same radius of convergence.

The dominant singularity of $F_k$ comes either from the zero of the denominator or from the $P_j$s, which is also that of the \(L_j\). If we can show that $F_k$ has the same radius of convergence as the $L_j$, then the dominant singularity of $F_k$ must come from the $P_j$s and so not be a pole. As a consequence, the asymptotics will have a similar form to that of the $L_j$.

Firstly, using the popularity argument \cite{hammersley1982} from statistical mechanics, we can show that the growth rates are the same for \(q > 0\). We sketch it out here for completeness. Unfortunately, it does not work for other \(q\) and hence we will need more specialised results.

\subsection{Popularity argument}
\label{sec:pop}
We will need the following lemma, which is a generalisation of the well-known Fekete's lemma by Erdős and de Bruijn \cite[Theorem~23]{fekete}.

\begin{lemma}[Fekete, Erdős, de Bruijn]
\label{lem:fekete_general}
  Consider a sequence \(\{a_n\}_n\), and a function \(\psi(t)\) that is positive and increasing for \(t > 0\). Assume that 
  \begin{align}
    a_{n+m} \leq a_n + a_m + \psi(n+m) \quad \text{ and } \quad\int_1^{\infty} \frac{\psi(t)}{t^2}\mathrm{d}t < \infty.
  \end{align}
  Then the following limit exists for some \(-\infty \leq \ell < \infty\):
  \begin{align}
    \lim_{n \longrightarrow \infty} \frac{a_n}{n} = \ell.
  \end{align}
\end{lemma}

Consider our original system \(\bm{Y}(z, q) = \bm{\Phi}(\bm{Y}, z, q)\). Since all the \(Y_i\) have the same radius of convergence, we choose \textit{any} \(L_0^{(i)}\) and call it \(L\). We will argue that it has the same growth rate as \(F_k\). So, we have 
\begin{align}
  F_k(z; q) = \sum_{n, m} f_{n, m} z^n q^m, && L(z; q) = \sum_{n, m} c_{n, m} z^n q^m, \\
  f_n = f_n(q) = \sum_m f_{n, m} q^m, && c_n = c_n(q) = \sum_m c_{n, m} q^m.
\end{align}
\begin{lemma}
\label{lem:feketelim}
Let \(q > 0\). Then, as functions of \(q\), the following limits exist:
\begin{align}
  \lim_{n \longrightarrow \infty} f_n^{1/n} \quad \text{ and } \quad \lim_{n \longrightarrow \infty} c_n^{1/n}.
\end{align}
\end{lemma}
\begin{proof}
We show this for \(f_n\), and the same argument works for \(c_n\) as well.

We see that \(f_{n_1 + n_2, m_1 + m_2} \geq f_{n_1,m_1}f_{n_2,m_2}\) since we can concatenate closed walks. The resulting walk is still closed. Now, for any \(n_1, n_2\), 
\begin{align}
  f_{n_1} f_{n_2} &= \sum_{m_1} f_{n_1, m_1} q^{m_1} \cdot \sum_{m_2} f_{n_2, m_2} q^{m_2} \\
  &= \sum_{m_1}\sum_{m_2} f_{n_1, m_1} f_{n_2, m_2}q^{m_1+m_2} \\
  &\leq \sum_{m_1}\sum_{m_2} f_{n_1 + n_2, m_1+m_2} q^{m_1+m_2} \\
  &\leq 2(n_1 + n_2)\sum_{m} f_{n_1 + n_2, m} q^m.
\end{align}
Here, the last inequality follows from the fact that we can sum over all possible values of \(m = m_1 + m_2\), and there are at most \(2(n_1+n_2)\) ways of writing \(m\) as a sum of \(m_1\) and \(m_2\), each having the same contribution to the sum.

This gives us \(f_{n_1} f_{n_2} \leq 2(n_1+n_2) f_{n_1 + n_2}\). So, we get that 
\begin{align}
  (-\log f_{n_1+n_2}) \leq (-\log f_{n_1}) + (-\log f_{n_2}) + \log(2(n_1+n_2)).
\end{align}
We also note that
\begin{align}
  \int_1^\infty \frac{\log (2t)}{t^2} \mathrm{d}t = 1 + \log 2 < \infty.
\end{align}
Hence, by Lemma~\ref{lem:fekete_general}, \(\log(f_n)/n \longrightarrow \ell \in (-\infty, \infty]\). However, \(\ell\) is not \(\infty\) since \(\log (f_n)/n\) is bounded above (in this case by \(2k\), since at every vertex we have \(2k\) choices of the next step). So the limit exists and is finite.
\end{proof}
\begin{theorem}
\label{thm:popularity}
  As a function of \(q\), the growth rate of \(L\) is the same as that of \(F_k\). That is, \(\mu_q (L) = \mu_q(F_k)\), when \(q > 0\).
\end{theorem}
\begin{proof}
  Since \(F_k\) counts all walks counted by \(L\), we see that for any \(n\)
\begin{align}
  c_n = \sum_m c_{n, m} q^m \leq \sum_m f_{n, m} q^m \leq n = f_n.
\end{align}
Hence, \(\limsup |c_n|^{1/n} \leq \limsup |f_n|^{1/n}\). So, \(\mu_q (L) \leq \mu_q(F_k)\) when \(q > 0\).

For the other direction, consider the walks counted by \(f_n\). We characterise them by the maximum distance the walk reaches from the root. Say \(f_n(h) = \) \# of closed walks of length \(n\) that reach a maximum of \(h\) distance away from the root. 

Every walk is \textit{weighted} by the winding number contribution, in the form of \(q^{m}\). Now, \(f_n(h) = 0\) when \(h > n/2\), since the walks are closed. So, \(f_n = \sum_{h=0}^{n/2} f_n(h)\). We choose an \(h^*\) such that \(f_{n}(h^*) \geq f_n(h)\) for all \(h\). This is called the \textit{most popular} choice. Hence, \(f_{n}(h^*) \leq f_n \leq n \cdot f_{n}(h^*)\).

We now relate \(f_n(h^*)\) to \(L\). Let \(y\) be an \textit{open} walk of length \(h^*\) that starts from the root, and walks to a vertex \(v\) that is \(h^*\) distance away. Since the vertex \(v\) is \(h^*\) distance away, a walk between the root and \(v\) of length \(h^*\) is necessarily the shortest. Let the reversal of this walk be \(\hat{y}\).

For any closed walk \(w\) (counted by \(F_k\)) that reaches a max distance of \(h^*\) from the root, we can make a new walk \(y \cdot w^M \cdot \hat{y}\), where \(w^M\) signifies repeating the walk \(M\) times. This is well defined since \(w\) is a closed walk (see Figure~\ref{fig:popularity}). This walk remains on one side of the graph by construction, and is hence counted by \(L\). So, it follows that
\begin{align}
  f_n^M &\leq n^M \cdot c_{nM + 2h^*}, \\
  f_n^{\frac{M}{nM + 2h^*}} &\leq n^{\frac{M}{nM + 2h^*}} \cdot c_{nM + 2h^*}^{\frac{1}{nM + 2h^*}}, \\
  f_n^{1/n} &\leq n^{1/n} \cdot \mu_q(L) && \text{(Taking \(M \longrightarrow \infty\)),} \\
  \mu_q(F_k) &\leq \mu_q(L) && \text{(Taking \(n \longrightarrow \infty\)).}
\end{align}
Here, the limits exist due to Lemma~\ref{lem:feketelim}.
\end{proof}

\begin{figure}[h]
  \centering
  \includegraphics[width=0.9\linewidth]{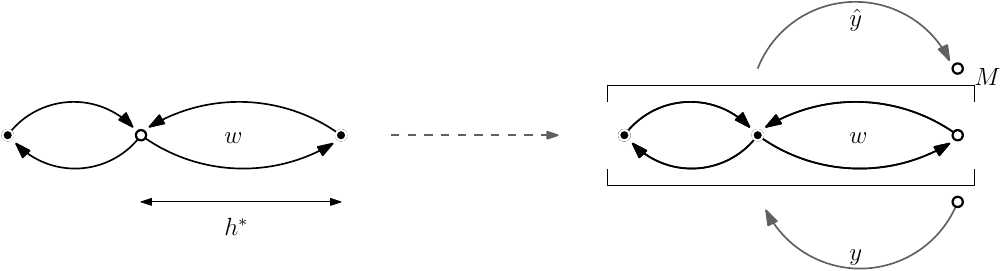}
  \caption{Construction of new walks in the popularity argument.}
  \label{fig:popularity}
\end{figure}

We note that the same argument also works between \([q^0]F_k(z; q)\) and \(F_k(z; 1)\) to show that they have the same growth rate (sketched below), which we will use in Corollary~\ref{cor:d_finite}.

\begin{lemma}\label{lem:growth_rate_pop}
\([q^0]F_k(z; q)\) and \(F_k(z; 1)\) have the same exponential growth rate.
\end{lemma}
\begin{proof}
Let the growth rate of \([q^0]F_k(z; q)\) be \(\mu_0\) and that of \(F_k(z; 1)\) be \(\mu_1\). Then, clearly \(0 \leq \mu_0 \leq \mu_1 \leq 2k\). In order to complete the proof, we show that $\mu_0 \geq \mu$. 

Again, $f_{n,m} = 0$ when $|m|> n/2$, and $f_{n,m} = f_{n,-m}$. So, we write
\begin{align}
 f_n &= \sum_{m=\lfloor-n/2\rfloor}^{\lceil n/2\rceil} f_{n,m}.
\end{align}
For a given $n$, there must be there is a value of $m$ which gives the largest summand. That is, define $m^*$ so that $f_{n,m^*} \geq f_{n,m}$ for all $m$. It follows that
\begin{align}
 f_{n,m^*} \leq f_n \leq (n+1) f_{n,m^*}.
\end{align}
Taking the $n$-th root and then the limit as $n \longrightarrow \infty$ gives
\begin{align}
 \limsup_{n \longrightarrow \infty} f_{n,m^*}^{1/n} &= \limsup_{n \longrightarrow \infty} f_n^{1/n}= \lim_{n \to
\infty} f_n^{1/n} = \mu.
\end{align}

Now fix $n$ and consider a walk counted by $f_{n,m^*}$ and another counted by $f_{n,-m^*}$. Concatenating them gives a walk counted by $f_{2n,0}$. Then considering all possible ways of concatenating $M$ walks counted by $f_{n,m^*}$ with $M$ walks counted by $f_{n,-m^*}$ gives the inequality
\begin{align}
 f_{n,m^*}^M f_{n,-m^*}^M \leq f_{n,m^*}^{2M} \leq f_{2nM,0}.
\end{align}
Now take the $(2nM)$-th root of both sides and let $M \longrightarrow \infty$ to obtain \(f_{n,k^*}^{1/n} \leq \mu_0\).

Letting $n \longrightarrow \infty$ then shows that $\mu_1\leq \mu_0$. Thus, we must have that $\mu_1 = \mu_0$.
\end{proof}

\subsection{Geometric group theory}
In order to proceed with showing that the growth rate of \(L\) and \(F_k\) is the same for all \(q\), we use tools from geometric group theory and probability. We proceed by arguing that \(F_k\) is not dominated by a pole.

 We first recall that the subgroup \(\langle \Delta \rangle\) of \(\mathcal{G}_k\) is normal (Lemma~\ref{lem:normal}), so the Schreier graph of \(\mathcal{G}_k\) with respect to the subgroup \(\langle \Delta \rangle\) is the same as the Cayley graph of the quotient group \(\langle \Delta \rangle\backslash \mathcal{G}_k\), which has the presentation
\begin{align}
  \langle a_1, \ldots, a_k \mid a_1^{p_1} = a_2^{p_2} = \cdots = a_k^{p_k} = e\rangle.
\end{align}
We denote the quotient group as $\Tilde{\mathcal{G}}_k$.

In light of this identification, we work with walks on the Cayley graph of \(\Tilde{\mathcal{G}}_k\) from now on. Note that \(\Tilde{\mathcal{G}}_k\) is a free product of finite groups, so it is virtually free (see, for example, \cite[Theorem 2]{Lyndon_1973}) and hence hyperbolic.

Recall that the case in which \(p_1=p_2=\dots=p_k=2\) has been handled explicitly in 
Section~\ref{sec:special_case}. Consequently we assume here that \(\max\{p_1,\dots,p_k\} \geq 3\).

A hyperbolic group is said to be \emph{non-elementary} if it is not virtually cyclic (that is, not virtually $\Z$ or finite). Notice that since we can assume that there exists \(p_\ell \geq 3\), for any \(i \neq \ell\) the subgroup $\langle a_\ell a_i, a_\ell^{-1} a_i \rangle $ is free of rank $2$, so the group $\tilde{\mathcal G}_k$ is non-elementary hyperbolic.

It follows that we may use the following result (\cite[Theorem~1.1]{hyperbolic}) on asymptotics of random walks on Cayley graphs of hyperbolic groups. 

\begin{theorem}[Gouëzel \cite{hyperbolic}]
  \label{thm:gouezel}
  Let \(\Gamma\) be a hyperbolic and non-elementary group. Consider the simple random walk on \(\Gamma\) starting at the root. Let \(P^{(n)}\) denote the probability that the walk is at the root after \(n\) steps. 
   Then \[P^{(n)} \sim C \rho^{-n} n^{-3/2},\] where \(C>0\) is some constant and \(\rho > 1\) is the inverse of the spectral radius.
  This holds for all \(n\) if the walk is not periodic. If the walk is periodic, the same still holds when restricted to even \(n\).
\end{theorem}

Note that \cite[Theorem~1.1]{hyperbolic} is stated for any admissible finitely supported symmetric probability measure on $G$, which for us corresponds to the simplest case of having uniform probability on each element of $S\cup S^{-1}$, which is a symmetric generating set for $\Gamma$.

\subsection{Expected number of returns}
In this section, we drop the dependence on \(q\), since we only talk about the structure and returns of the walk, without keeping track of the winding number. This, henceforth, is agnostic of the value of \(q\). This is precisely the reason we are using this argument; arguments depending on the value of \(q\) seem not to work well (see Section~\ref{sec:pop}). 

\begin{defn}
  Let \((X_n)_n\) be a simple random walk on the Cayley graph of \(\Tilde{\mathcal{G}}_k\) starting at the root (denoted by \(0\)).
\end{defn}

As noted above, this is the same as the walk on the Schreier graph of \(\mathcal{G}_k\). The walk is, unlike the walks considered above, not constrained to end at the root yet. 
\begin{defn}
  Let \(P^{(n)}\) be the probability that after \(n\) steps, the walk is at the root. Hence, 
\begin{align}
  P^{(n)} = \mathbb P[X_n = 0 \mid X_0 = 0].
\end{align}
\end{defn}
We note here that in the notation of our sequences of walks, we have 
\begin{align}
  P^{(n)} = \frac{f_n}{(2k)^n}, && \rho = \frac{1}{\limsup_n (P^{(n)})^{1/n}} = \frac{2k}{\mu(F_k)}.
\end{align}
Since \(f_{n+m} \geq f_n f_m\), we have \(P^{(n+m)} \geq P^{(n)}P^{(m)}\).

\begin{defn}
  Let \(V_n\) be the number of visits to the origin in a random walk of length \(n\), and let \(v_n = \mathbb E [V_n \mid X_n = 0]\).
\end{defn}
Consider such a visit in a walk; by cutting the walk at that point, one obtains two shorter walks that start and end at the root. Hence, 
\begin{align}
  v_n = \mathbb E [V_n \mid X_n = 0] = \dfrac{\sum_{i=0}^{n-1} P^{(i)}P^{(n-i)}}{P^{(n)}},
\end{align}
when \(P^{(n)} \neq 0\). If it is \(0\), we set the expectation to be \(0\) as well.
The denominator here follows since we are conditioning on a bridge, so we divide by that probability. 

\begin{theorem}
\label{thm:returns}
  Restricted to the subsequence where \(P^{(n)} \neq 0\), \(v_n\) is uniformly bounded.
\end{theorem}
\begin{proof}
  From Theorem~\ref{thm:gouezel}, we get that \(P^{(n)} \sim C \rho^{-n} n^{-3/2}\) for some constant \(C\). Let \(N \in \N\) such that for any \(n > N\)
\begin{align}
  \left |\frac{P^{(n)}}{C \rho^{-n} n^{-3/2}} - 1 \right | \leq \frac{1}{2}.
\end{align}
Then we have for \(n > N\):
\begin{align}
  v_n &= \dfrac{\sum_{i=0}^{n-1} P^{(i)}P^{(n-i)}}{P^{(n)}} 
  \leq 2\dfrac{\sum_{i=0}^{n/2} P^{(i)}P^{(n-i)}}{P^{(n)}} \\
  &= \dfrac{2}{P^{(n)}} \left (\sum_{i = 0}^{N}P^{(i)}P^{(n-i)}+ \sum_{i = N+1}^{ n/2 }P^{(i)}P^{(n-i)} \right ) \\
  &\leq 2 (N+1) + \frac{2}{P^{(n)}}\sum_{i = N+1}^{ n/2 }P^{(i)}P^{(n-i)} && \text{since } P^{(n)} \geq P^{(i)}P^{(n-i)} \\
  &\leq 2 (N+1) + \frac{9C}{n^{-3/2}} \sum_{i = N+1}^{ n/2 }i^{-3/2}(n-i)^{-3/2} \\
  &\leq 2 (N+1) + \frac{9C}{n^{-3/2}} \sum_{i = N+1}^{ n/2 }i^{-3/2}\left (\frac{n}{2} \right )^{-3/2} \\
  &\leq 2 (N+1) + 18\sqrt{2}C\sum_{i = 1}^{\infty}i^{-3/2} = 2 (N+1) + 18\sqrt{2}C \cdot \zeta(3/2),
\end{align}
where \(\zeta(3/2) \approx 2.61237\ldots\). Hence, \(v_n\) is bounded independent of \(n\).
\end{proof}

\subsection{Return to the singularity}

We use the expected number of returns to demonstrate that the dominant singularity is not a simple pole. The following lemma shows how to link the expected number of returns in a walk and the associated generating functions. It follows from fairly standard arguments.

\begin{lemma}
\label{lem:returns_combinatorial}
  Consider combinatorial classes \(\mathcal{A}, \mathcal{B}\) with associated generating fucntions \(A(z) = \sum_n \alpha_i z^n\) and \(B(z) = \sum_n \beta_i z^n\) respectively. Let \(\mathcal{A}\) be obtained by concatenating elements from \(\mathcal{B}\) together in sequences. That is, their generating functions are related as
  \begin{align}
    A(z) = 1 + B(z) + B(z)^2 + \cdots = \frac{1}{1 - B(z)}.
  \end{align}
  In a random element of \(\mathcal{A}\) of size \(n\), let the number of elements of \(\mathcal{B}\) be counted by a random variable \(T\). Then
  \begin{align}
    \mathbb E [T] = \dfrac{\sum_{i=1}^n \alpha_i \alpha_{n-i}}{\alpha_n}.  
  \end{align}
\end{lemma}
\begin{proof}
  We construct an auxiliary generating function 
  \begin{align}
    \hat{A}(z; t) = 1 + tB(z) + t^2B(z)^2+ \cdots = \frac{1}{1-tB(z)}.
  \end{align}
  Hence, the variable \(t\) counts the occurrences of elements of \(\mathcal{B}\). We can compute the number of occurrences of \(\mathcal{B}\)'s by differentiating with respect to \(t\). This gives 
  \begin{align}
    \frac{\partial \hat{A}}{\partial t} \bigg|_{t=1} = \hat{A}(z;1)^2-\hat{A}(z;1).
  \end{align}
  Then, applying \cite[Proposition~III.2]{Flajolet} and setting \(t=1\) we get
  \begin{align}
    \mathbb E [T] = \dfrac{[z^n](A(z)^2 - A(z))}{[z^n]A(z)} = \dfrac{\sum_{i=0}^n \alpha_i \alpha_{n-i}}{\alpha_n} - 1
    = \dfrac{\sum_{i=1}^n \alpha_i \alpha_{n-i}}{\alpha_n}
  \end{align}
  since \(\hat{A}(z; 1) = A(z)\).
\end{proof}

We now prove the main technical result of the paper, which gives us a local limit law for \(F_k\). 

\begin{theorem}
\label{thm:main}
  For the group \(\mathcal{G}_k\) with at least one \(p_i > 2\), the dominant asymptotics of \(F_k(z; q)\) come from the dominant asymptotics of \(L(z; q)\). Consequently, 
  \begin{align}
    [z^{2n}q^{2m}] F_k(z; q) = f_{2n, 2m} =C \rho^{-n} n^{-2} \left (\exp\left (\frac{-m^2}{\gamma n} \right ) + \mathcal{O}(\sqrt{n}) \right),
  \end{align}
  when \(m\) is close to \(0\), and \(C, \rho, \gamma\) are positive constants.
\end{theorem}
\begin{proof}
  We begin by noting again that
  \begin{align}
    F_k(z; q) = \frac{1}{1-\sum_i P_i(z; q)}.
  \end{align}
  As discussed at the start of Section~\ref{sec:asymp2}, the dominant singularity of $F_k$ comes either from the zero of the denominator or from the $P_i$'s, which is also that of the \(L\). We now argue that the dominant singularity of \(F_k\) does \textit{not} come from a pole. As a consequence, the asymptotics will have a similar form to that of the $L_j$.

  Suppose the dominant singularity (say \(z = r\)) does come from a pole at the zero of the denominator. Then, locally around \(z=r\), one can express
  \begin{align}
    F_k(z) = Q_1(z) + \frac{Q_2(z)}{\left ( 1 - z/r\right )^{\ell}}
  \end{align}
  where \(\ell \in \{1, 2, 3, \ldots\}\) is the order of the pole, and \(Q_1, Q_2\) are analytic functions.

  Then, standard analytic combinatorics results \cite{Flajolet} show that \(f_n \sim C r^{-n} n^{\ell - 1}\).

  By Lemma~\ref{lem:returns_combinatorial}, one gets the expected number of returns:
  \begin{align}
    \mathbb E [T] = \frac{\sum_{i=1}^n f_i f_{n-i}}{f_n} - 1 \geq \hat{C} \frac{\sum_{i=1}^n i^{\ell - 1}}{n^{\ell - 1}} = \Omega(n),
  \end{align}
  for some positive constant \(\hat{C}\). This can be made precise, similar to the proof of Theorem~\ref{thm:returns}.
  
  Similarly, if the dominant singularity comes from confluent singularities from the pole \textit{and} the singularity of \(P\) at the same radius, we get 
  \begin{align}
    F_k(z) = Q_1(z) + \frac{Q_2(z)}{\left ( 1 - z/r\right )^{\ell-1/2}},
  \end{align}
  where \(\ell \in \{1, 2, 3, \ldots\}\) is the order of the pole, and \(Q_1, Q_2\) are analytic functions. This happens since \(P_i\)s and \(L\) have a square-root singularity.
  
  The same argument as above shows that the expected number of returns is \(\Omega(\sqrt{n})\). Both these cases contradict Theorem~\ref{thm:returns}.

  As a consequence, we see that the dominant singularity of \(F_k\) is the same as that of \(L\), and the result follows.
\end{proof}

We can now put everything together to prove that the cogrowth series of groups in our infinite family are D-finite and not algebraic.

\begin{cor}[Cogrowth series for $\mathcal G(p_1, \ldots, p_k)=\mathcal G_k$]
  \label{cor:d_finite} 
  For the group \(\mathcal{G}_k\), the cogrowth series is D-finite but not algebraic. Additionally, the cogrowth series is an algebraic number.
\end{cor}
\begin{proof}
When \(p_1=p_2=\cdots=p_k=2\), the result follows from Theorem~\ref{thm:k_tree_expansion}. So now assume that at \(\max\{p_1,\cdots,p_k\}\geq 3\), so that we can use the asymptotic results from Theorem~\ref{thm:main}.

Note $F_k(z;q)$ is algebraic (solution to a system of algebraic equations), and hence D-finite. As a result, \([q^0]F_k(z;q)\) is D-finite \cite{MR929767}, and hence the cogrowth series is D-finite. This is not algebraic since the $n^{-2}$ factor in the asymptotics is incompatible with an algebraic generating function (\cite[Section VII.7]{Flajolet} and \cite{jungen}).

In order to see that the cogrowth is an algebraic number, one uses Lemma~\ref{lem:growth_rate_pop}. Then, since \(F_k(z; 1)\) is an algebraic generating function, the cogrowth is an algebraic number.
\end{proof}

\section{Applying the techniques to other presentations}
\label{sec:other_cases}
In this section, we consider three presentations for the 3-strand braid group as shown in Equation~\eqref{eq:B3preses}.

Starting with $\langle a,b \mid aba=bab \rangle$
and applying the Tietze transformation (see \cite[Proposition 2.1]{LS}) that adds a new generator $x$ we obtain 
\begin{align}
  \langle a,b, x \mid aba=bab, x=ba \rangle=\langle a,b, x \mid abaa=baba, x=ba \rangle=\langle a, b, x \mid axa=x^2, xa^{-1}=b \rangle.
\end{align}
Then removing the generator $b$ yields $\langle a, x \mid axa=x^2\rangle$.

Now adding a new generator $c$ we obtain 
 \begin{align}
   \langle a, x,c \mid axa=x^2, c=ax\rangle=\langle a, x,c \mid axax=x^3, c=ax\rangle=\langle a, x,c \mid c^2=x^3, cx^{-1}=a\rangle.
 \end{align}
Then removing $a$ yields $\langle x,c \mid c^2=x^3 \rangle$. Note that this presentation is also referred to as the (torus or) trefoil knot presentation, since it is a standard presentation for the fundamental group of the trefoil knot complement.

\subsection{The standard presentation} 
Here, we work with the presentation
\begin{align}
 \langle a,b \mid aba = bab \rangle
\end{align} and apply our results above to this case. As before, we start by choosing a special element \(\Delta\). Previously, we chose a central element for \(\Delta\), but in this case, we set \(\Delta = aba \equiv bab\); this does not commute with the generators but does interact with them in a simple way as follows.
\begin{align}
 \Delta a &\equiv b \Delta, & 
 \Delta b &\equiv a \Delta, \\
 \Delta a^{-1} &\equiv ba, &
 \Delta b^{-1} & \equiv ab.
\end{align}
Thus $\Delta^2$ is central; we could replace $\Delta$ by $\Delta^2$. However, it makes the subsequent work more cumbersome. Some fairly standard work shows that every element of the group can be represented by a word with normal form
\begin{align}
 w \equiv \Delta^m p, 
\end{align}
where $p$ is a word in non-negative powers of the generators $\{a,b\}^*$. Again, we form the Schreier graph of the cosets of the subgroup generated by $\Delta$. The graph forms a tree (see Figure~\ref{fig aba cosests}) similar to the group \(\langle a,b \mid a^3=b^3\rangle \) discussed in Section~\ref{sec:gen_funcs_walks}. Due to this similarity in the Schreier graphs, similar and quite explicit generating function arguments work here. We give the details in Appendix~\ref{appx:braid}. It turns out that the cogrowth series of this group with respect to this generating set is the same as that of \(\mathcal{G}(3,3)\).

\begin{theorem}
  The cogrowth series for $B_3$ with respect to the generating set $\{a,b\}$ is D-finite but not algebraic. 
\end{theorem}

\begin{figure}[h!]
\centering
 \includegraphics[width=0.45\textwidth]{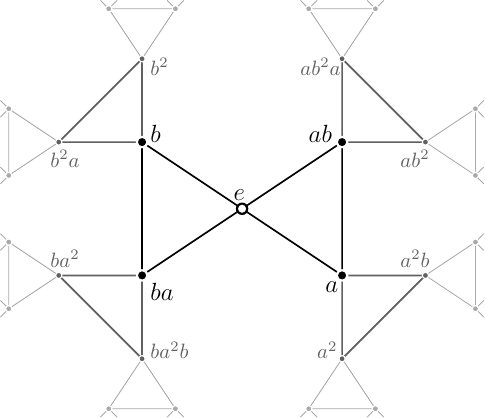}
 \caption{The Schreier graphs of the cosets of the subgroup $\langle \Delta \rangle$ of $B_3$ with respect to the generating set $\{a,b\}$. The
vertex labelled $p$ corresponds to the coset $\langle \Delta \rangle p$. }
\label{fig aba cosests}
\end{figure}

We note that our methods can be extended to generalisations of the group presented by 
\begin{align}
 \langle a,b \mid (a,b)_{p} = (b,a)_p \rangle,
\end{align}
where $p\geq 3$ and $\Delta = (a,b)_{p} = \underbrace{abab\dots }_{p \text{ symbols}}$.

The Schreier graph consists of polygons with side lengths \(p\), which means this group will have the same cogrowth series as \(\mathcal{G}(p,p)\) and so be D-finite and not algebraic.

This extends in a reasonably straightforward way to any number of generators. For example, with 3 generators, the group with presentation 
\begin{align}
 \langle a,b,c \mid (a,b,c)_{p} = (b,c,a)_p = (c,a,b)_p \rangle,
\end{align}
where $\Delta = (a,b,c)_{p} = \underbrace{abcabc\dots }_{p \text{ symbols}}$, has the same cogrowth series as \(\mathcal{G}(p,p,p)\). 

\subsection{Second presentation}
Next, we consider the presentation \(\langle x, c \mid c^2 = x^3\rangle\) for $B_3$. This is \(\mathcal{G}(2,3)\), and so we know that its cogrowth series is D-finite but not algebraic. In Appendix~\ref{appx:torus_knot} we compute the explicit asymptotic behaviour of the cogrowth series with respect to this generating set.

\subsection{Third presentation}
\label{subsec_axa}
Now consider the presentation $\langle a, x \mid axa=x^2\rangle$ for $B_3$.
Setting $\Delta = x^3$, the Schreier graph with respect to $\langle \Delta\rangle$ and $\{a,x\}$ is shown in Figure~\ref{fig:axa_1}. This still gives us a tree-like graph of triangles and squares, but now these meet along edges rather than vertices. 

\begin{figure}[h!]
\begin{center}
 \includegraphics[width = 0.5\textwidth]{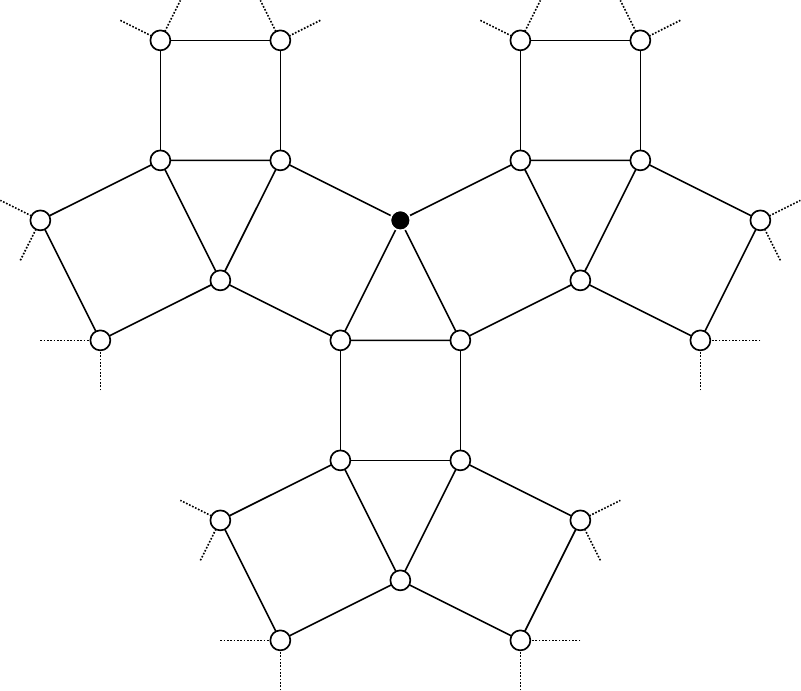}
\end{center}
\caption{The Schreier graph of $B_3$ with respect to the presentation \(\langle a, x \mid axa = x^2 \rangle\) and $\Delta = x^3$. The coset $\langle \Delta \rangle$ is indicated by the black vertex.}
\label{fig:axa_1}
\end{figure}

The absence of cut-vertices renders the combinatorial analysis substantially more difficult. See Appendix~\ref{appx:axa_group}. In spite of this, we can still show that the cogrowth series is D-finite but not algebraic.
\begin{theorem}\label{thm:axa-cogrowth-Dfinite}
  The cogrowth series for $B_3$ with respect to the generating set $\{a,x\}$ is D-finite but not algebraic. 
  \label{thm:axa_thm}
\end{theorem}

The proof is quite similar to the results above, but requires us to keep track of more tedious details. In particular, we show that the cogrowth series is D-finite by (as done above) finding a system of algebraic equations with the cogrowth series as its diagonal. As a consequence, we reserve the details of proving D-finiteness to Appendix~\ref{appx:axa_group} (Theorem~\ref{thm:cogrowth-details-axa}).

The proof for non-algebraicity follows from similar arguments on asymptotics as described in the main body of the paper. The steps are similar but more tedious. In particular, more care is required to bound the number of returns and hence show that the singularities are the same. However, one can work this out to show that this group presentation also has non-algebraic cogrowth series.

\usection{Acknowledgements}
The authors thank Omer Angel, Tony Guttmann, Kasia Jankiewicz, Manuel Kauers, Jens Malmquist, and Stephan Wagner for helpful conversations. Elder acknowledges funding in the form of an ARC Discovery Project. Rechnitzer acknowledges funding in the form of an NSERC Discovery Project.

\bibliographystyle{plain}
\bibliography{arxiv_refs}

@article{Lyndon_1973, 
title={Two notes on {R}ankin’s book on the modular group}, 
volume={16}, 
DOI={10.1017/S1446788700015433}, 
number={4}, 
journal={J. Aust. Math. Soc.}, 
fjournal={Journal of the Australian Mathematical Society}, author={Lyndon, Roger}, 
year={1973}, 
pages={454–457}}

@article {GrigGCC,
    AUTHOR = {Darbinyan, Arman and Grigorchuk, Rostislav and Shaikh, Asif},
     TITLE = {Finitely generated subgroups of free groups as formal
              languages and their cogrowth},
   JOURNAL = {J. Groups Complex. Cryptol.},
  FJOURNAL = {Journal of Groups, Complexity, Cryptology},
    VOLUME = {13},
      YEAR = {2021},
    NUMBER = {2},
     PAGES = {Paper No. 1, 30},
      ISSN = {1867-1144},
   MRCLASS = {20F10 (20E05 20E07 20F05)},
  MRNUMBER = {4349469},
       DOI = {10.46298/jgcc.2021.13.2.7617},
       URL = {https://doi.org/10.46298/jgcc.2021.13.2.7617},
}

@article {grigMultiCogrowth,
    AUTHOR = {Grigorchuk, R. and Quint, J.-F. and Shaikh, A.},
     TITLE = {Multivariate growth and cogrowth},
   JOURNAL = {Carpathian Math. Publ.},
  FJOURNAL = {Carpathian Mathematical Publications. Karpat\cdot s\cprime
              k\=\i\ Matematichn\=\i\ Publ\=\i kats\=\i\"i},
    VOLUME = {17},
      YEAR = {2025},
    NUMBER = {1},
     PAGES = {82--109},
      ISSN = {2075-9827,2313-0210},
   MRCLASS = {20F69 (05A05 05A15 20E05 60F10 68Q45)},
  MRNUMBER = {4917829},
MRREVIEWER = {Valery\ A.\ Liskovets},
}

@article {Humphries,
    AUTHOR = {Humphries, Stephen P.},
     TITLE = {Generalised cogrowth series, random walks, and the group
              determinant},
   JOURNAL = {Math. Proc. Cambridge Philos. Soc.},
  FJOURNAL = {Mathematical Proceedings of the Cambridge Philosophical
              Society},
    VOLUME = {165},
      YEAR = {2018},
    NUMBER = {3},
     PAGES = {445--465},
      ISSN = {0305-0041,1469-8064},
   MRCLASS = {20F05 (20C15 20P05 60G50)},
  MRNUMBER = {3860398},
MRREVIEWER = {Geetha\ Venkataraman},
       DOI = {10.1017/S0305004117000573},
       URL = {https://doi-org.ezproxy.lib.uts.edu.au/10.1017/S0305004117000573},
}

@article {PriceGuttmann,
    AUTHOR = {Price, Andrew Elvey and Guttmann, Anthony J.},
     TITLE = {Numerical studies of {T}hompson's group {$F$} and related
              groups},
   JOURNAL = {Internat. J. Algebra Comput.},
  FJOURNAL = {International Journal of Algebra and Computation},
    VOLUME = {29},
      YEAR = {2019},
    NUMBER = {2},
     PAGES = {179--243},
      ISSN = {0218-1967,1793-6500},
   MRCLASS = {20F65 (05A05 05A15 20F05)},
  MRNUMBER = {3934784},
MRREVIEWER = {J\"org\ Lehnert},
       DOI = {10.1142/S0218196718500686},
       URL = {https://doi-org.ezproxy.lib.uts.edu.au/10.1142/S0218196718500686},
}

@article {bishopDiagonal,
    AUTHOR = {Bishop, Alex},
     TITLE = {On groups whose cogrowth series is the diagonal of a rational
              series},
   JOURNAL = {Internat. J. Algebra Comput.},
  FJOURNAL = {International Journal of Algebra and Computation},
    VOLUME = {34},
      YEAR = {2024},
    NUMBER = {8},
     PAGES = {1209--1224},
      ISSN = {0218-1967,1793-6500},
   MRCLASS = {20F65 (05A15 20K35 68R15)},
  MRNUMBER = {4846866},
MRREVIEWER = {Enric\ Ventura Capell},
       DOI = {10.1142/S0218196724500486},
       URL = {https://doi-org.ezproxy.lib.uts.edu.au/10.1142/S0218196724500486},
}

@article{bell2021cogrowth,
author = {Bell, Jason and Liu, Haggai and Mishna, Marni},
title = {Cogrowth series for free products of finite groups},
journal= {Internat. J. Algebra Comput. },
fjournal = {International Journal of Algebra and Computation},
volume = {33},
number = {02},
pages = {237-260},
year = {2023},
doi = {10.1142/S0218196723500133}
}

@article {BellMishnaAmenable,
    AUTHOR = {Bell, Jason and Mishna, Marni},
     TITLE = {On the complexity of the cogrowth sequence},
   JOURNAL = {J. Comb. Algebra},
  FJOURNAL = {Journal of Combinatorial Algebra},
    VOLUME = {4},
      YEAR = {2020},
    NUMBER = {1},
     PAGES = {73--85},
      ISSN = {2415-6302},
   MRCLASS = {20F05 (05C81 68Q45)},
  MRNUMBER = {4073891},
MRREVIEWER = {Joshua Maglione},
       DOI = {10.4171/jca/39},
       URL = {https://doi.org/10.4171/jca/39},
}

@article {GPak,
    AUTHOR = {Garrabrant, Scott and Pak, Igor},
     TITLE = {Words in linear groups, random walks, automata and
              {P}-recursiveness},
   JOURNAL = {J. Comb. Algebra},
  FJOURNAL = {Journal of Combinatorial Algebra},
    VOLUME = {1},
      YEAR = {2017},
    NUMBER = {2},
     PAGES = {127--144},
      ISSN = {2415-6302},
   MRCLASS = {20F05 (05C81 68Q45)},
  MRNUMBER = {3634780},
       DOI = {10.4171/JCA/1-2-1},
       URL = {https://doi.org/10.4171/JCA/1-2-1},
}

@article {Kukseries,
    AUTHOR = {Kuksov, Dmitri},
     TITLE = {Cogrowth series of free products of finite and free groups},
   JOURNAL = {Glasg. Math. J.},
  FJOURNAL = {Glasgow Mathematical Journal},
    VOLUME = {41},
      YEAR = {1999},
    NUMBER = {1},
     PAGES = {19--31},
      ISSN = {0017-0895},
   MRCLASS = {20F69 (20E06)},
  MRNUMBER = {1689726},
MRREVIEWER = {Andrea Sambusetti},
       DOI = {10.1017/S001708959997026X},
       URL = {http://dx.doi.org/10.1017/S001708959997026X},
}

@article {Kukrational,
    AUTHOR = {Kouksov, Dmitri},
     TITLE = {On rationality of the cogrowth series},
   JOURNAL = {Proc. Amer. Math. Soc.},
  FJOURNAL = {Proceedings of the American Mathematical Society},
    VOLUME = {126},
      YEAR = {1998},
    NUMBER = {10},
     PAGES = {2845--2847},
      ISSN = {0002-9939},
   MRCLASS = {20F05},
  MRNUMBER = {1487319},
MRREVIEWER = {Fabio Scarabotti},
       DOI = {10.1090/S0002-9939-98-04741-8},
       URL = {http://dx.doi.org/10.1090/S0002-9939-98-04741-8},
}

@article {ERRW,
    AUTHOR = {Elder, Murray and Rechnitzer, Andrew and Janse van Rensburg,
              Esaias J. and Wong, Thomas},
     TITLE = {The cogrowth series for {${\rm BS}(N, N)$} is {D}-finite},
   JOURNAL = {Internat. J. Algebra Comput.},
  FJOURNAL = {International Journal of Algebra and Computation},
    VOLUME = {24},
      YEAR = {2014},
    NUMBER = {2},
     PAGES = {171--187},
      ISSN = {0218-1967},
   MRCLASS = {20F69 (05A15 20F65)},
  MRNUMBER = {3192369},
MRREVIEWER = {Xiangdong Xie},
       DOI = {10.1142/S0218196714500106},
       URL = {http://dx.doi.org/10.1142/S0218196714500106},
}

@article {GLalley,
    AUTHOR = {Gou\"ezel, S\'ebastien and Lalley, Steven P.},
     TITLE = {Random walks on co-compact {F}uchsian groups},
   JOURNAL = {Ann. Sci. \'Ec. Norm. Sup\'er. (4)},
  FJOURNAL = {Annales Scientifiques de l'\'Ecole Normale Sup\'erieure. Quatri\`eme
              S\'erie},
    VOLUME = {46},
      YEAR = {2013},
    NUMBER = {1},
     PAGES = {129--173 (2013)},
      ISSN = {0012-9593},
   MRCLASS = {60B15 (30F35 31C45 60G50 60J50)},
  MRNUMBER = {3087391},
MRREVIEWER = {Wolfgang Woess},
}

@book {Stanley,
    AUTHOR = {Stanley, R. P.},
     TITLE = {Enumerative combinatorics. {V}olume 1},
    SERIES = {Cambridge Stud. Adv. Math.},
    VOLUME = {49},
   EDITION = {Second},
 PUBLISHER = {Cambridge Univ. Press},
   ADDRESS = {Cambridge},
      YEAR = {2012},
     PAGES = {xiv+626},
      ISBN = {978-1-107-60262-5},
   MRCLASS = {05-02 (05A15 06-02)},
  MRNUMBER = {2868112},
}

@book {Stanley2,
    AUTHOR = {Stanley, R. P.},
     TITLE = {Enumerative combinatorics. {V}olume 2},
    SERIES = {Cambridge Stud. Adv. Math.},
    VOLUME = {62},
 PUBLISHER = {Cambridge Univ. Press},
   ADDRESS = {Cambridge},
      YEAR = {1999},
     PAGES = {xii+581},
      ISBN = {0-521-56069-1; 0-521-78987-7},
   MRCLASS = {05A15 (05-02 05E05 05E10 68R05)},
  MRNUMBER = {1676282 (2000k:05026)},
MRREVIEWER = {Ira Gessel},
       DOI = {10.1017/CBO9780511609589},
       URL = {http://dx.doi.org/10.1017/CBO9780511609589},
}

@article {Cohen,
    AUTHOR = {Cohen, Joel M.},
     TITLE = {Cogrowth and amenability of discrete groups},
   JOURNAL = {J. Funct. Anal.},
  FJOURNAL = {Journal of Functional Analysis},
    VOLUME = {48},
      YEAR = {1982},
    NUMBER = {3},
     PAGES = {301--309},
      ISSN = {0022-1236},
     CODEN = {JFUAAW},
   MRCLASS = {43A07},
  MRNUMBER = {678175 (85e:43004)},
MRREVIEWER = {P. Gerl},
       DOI = {10.1016/0022-1236(82)90090-8},
       URL = {http://0-dx.doi.org.library.newcastle.edu.au/10.1016/0022-1236(82)90090-8},
}

@book{Flajolet,
  title={{Analytic combinatorics}},
  author={Flajolet, P. and Sedgewick, R.},
  isbn={0521898064},
  year={2009},
  publisher={Cambridge Univ. Press}
}

@incollection {Grig,
    AUTHOR = {Grigorchuk, R. I.},
     TITLE = {Symmetrical random walks on discrete groups},
 BOOKTITLE = {Multicomponent random systems},
    SERIES = {Adv. Probab. Related Topics},
    VOLUME = {6},
     PAGES = {285--325},
 PUBLISHER = {Dekker},
   ADDRESS = {New York},
      YEAR = {1980},
   MRCLASS = {60B15 (20F05 43A07 60J15)},
  MRNUMBER = {599539 (83k:60016)},
MRREVIEWER = {Philippe Bougerol},
}

@article {ThompsonCogrowthERW,
    AUTHOR = {Elder, Murray and Rechnitzer, Andrew and Wong, Thomas},
     TITLE = {On the cogrowth of {T}hompson's group {$F$}},
   JOURNAL = {Groups Complex. Cryptol.},
  FJOURNAL = {Groups. Complexity. Cryptology},
    VOLUME = {4},
      YEAR = {2012},
    NUMBER = {2},
     PAGES = {301--320},
      ISSN = {1867-1144},
   MRCLASS = {20F65},
  MRNUMBER = {3043436},
}

@ARTICLE{Dunwoody85,
  author = {Dunwoody, Martin J.},
  title = {The accessibility of finitely presented groups},
  journal = {Invent. Math.},
  fjournal = {Inventiones Mathematicae},
  year = {1985},
  volume = {81},
  pages = {449--457},
  affiliation = {Mathematics Division University of Sussex BN1 9QH GB-Brighton UK},
  issn = {0020-9910},
  number = {3},
  publisher = {Springer Berlin / Heidelberg},
  url = {http://dx.doi.org/10.1007/BF01388581}
}

@incollection {ChomS,
    AUTHOR = {Chomsky, N. and Sch\"utzenberger, M. P.},
     TITLE = {The algebraic theory of context-free languages},
 BOOKTITLE = {Computer programming and formal systems},
     PAGES = {118--161},
 PUBLISHER = {North-Holland, Amsterdam},
      YEAR = {1963},
   MRCLASS = {94.50 (68.00)},
  MRNUMBER = {0152391},
MRREVIEWER = {E. Shamir},
}

@Article{ms83,
  author =       "David E. Muller and Paul E. Schupp",
  title =        "Groups, the theory of ends, and
                  context-free languages",
  pages =        "295--310",
  journal =      {J. Comput. System Sci. },
  volume =       "26",
  year =         "1983",
}

@book {LS,
    AUTHOR = {Lyndon, Roger C. and Schupp, Paul E.},
     TITLE = {Combinatorial group theory},
      NOTE = {Ergebnisse der Mathematik und ihrer Grenzgebiete, Band 89},
 PUBLISHER = {Springer-Verlag},
   ADDRESS = {Berlin},
      YEAR = {1977},
     PAGES = {xiv+339},
      ISBN = {3-540-07642-5},
   MRCLASS = {20F05 (55A05)},
  MRNUMBER = {MR0577064 (58 \#28182)},
MRREVIEWER = {Ian M. Chiswell},
}

@article{stanley1980,
  title={Differentiably finite power series},
  author={Stanley, R. P.},
  journal={European J. Combin},
  volume={1},
  number={2},
  pages={175--188},
  year={1980}
}

@article{hammersley1982,
author={Hammersley, J.M.  and Torrie, G.M. and  Whittington, S.G.},
title={Self-avoiding walks interacting with a surface},
journal={J. Phys. A},
fjournal={Journal of Physics A: Mathematical and General},
volume={15},
number={2},
pages={539},
url={http://stacks.iop.org/0305-4470/15/i=2/a=023},
year={1982}
}

@article {MR929767,
    AUTHOR = {Lipshitz, L.},
     TITLE = {The diagonal of a {$D$}-finite power series is {$D$}-finite},
   JOURNAL = {J. Algebra},
  FJOURNAL = {Journal of Algebra},
    VOLUME = {113},
      YEAR = {1988},
    NUMBER = {2},
     PAGES = {373--378},
      ISSN = {0021-8693},
     CODEN = {JALGA4},
   MRCLASS = {13F25},
  MRNUMBER = {929767 (89c:13027)},
MRREVIEWER = {Nica Vasile},
       DOI = {10.1016/0021-8693(88)90166-4},
       URL = {http://dx.doi.org/10.1016/0021-8693(88)90166-4},
}

@article{drmota1997,
author = {Drmota, Michael},
title = {Systems of functional equations},
journal = {Random Structures Algorithms},
volume = {10},
number = {1-2},
pages = {103-124},
doi = {https://doi.org/10.1002/(SICI)1098-2418(199701/03)10:1/2<103::AID-RSA5>3.0.CO;2-Z},
year = {1997}
}

@article{jungen,
    author = {R. Jungen},
    title = {Sur les séries de Taylor n'ayant que des singularités algébrico-logarithmiques sur leur cercle der convergence},
    journal = {Comment. Math. Helv.},
    fjournal = {Commentarii mathematici Helvetici},
    volume = {3},
    year = {1931},
    pages = {266-306}
}

@article{chapman17,
doi = {10.1088/1751-8121/aa6e45},
url = {https://dx.doi.org/10.1088/1751-8121/aa6e45},
year = {2017},
month = {May},
publisher = {IOP Publishing},
volume = {50},
number = {22},
pages = {225001},
author = {Harrison Chapman},
title = {Asymptotic laws for random knot diagrams},
journal = {J. Phys. A},
fjournal = {Journal of Physics A: Mathematical and Theoretical}
}

@article{bendergaorich,
title = {{Submaps of maps. I. General 0–1 laws}},
journal={J. Combin. Theory Ser. B },
fjournal = {Journal of Combinatorial Theory, Series B},
volume = {55},
number = {1},
pages = {104-117},
year = {1992},
issn = {0095-8956},
doi = {https://doi.org/10.1016/0095-8956(92)90034-U},
url = {https://www.sciencedirect.com/science/article/pii/009589569290034U},
author = {Edward A Bender and Zhi-Cheng Gao and L Bruce Richmond}
}

@article{selfavoiding,
doi = {10.1088/0305-4470/26/19/002},
url = {https://dx.doi.org/10.1088/0305-4470/26/19/002},
year = {1993},
month = {Oct},
publisher = {IOP Publishing},
volume = {26},
number = {19},
pages = {L981},
author = {E J Jance van Rensburg and  E Orlandini and  D W Sumners and  M C Tesi and  S G Whittington},
title = {The writhe of a self-avoiding polygon},
journal={J. Phys. A},
fjournal = {Journal of Physics A: Mathematical and General}
}

@book{ross,
  title={Stochastic Processes},
  author={Ross, S.M.},
  isbn={9780471120629},
  lccn={82008619},
  series={Wiley series in probability and mathematical statistics},
  year={1995},
  publisher={Wiley}
}

@book{drmota2009random,
  title={Random Trees: An Interplay between Combinatorics and Probability},
  author={Michael Drmota},
  isbn={9783211753576},
  series={Mathematics and Statistics},
  url={https://books.google.ca/books?id=JNJv7R1q9eIC},
  year={2009},
  publisher={Springer Vienna}
}

@article{scalinglimits22,
title = {Scaling limits of permutation classes with a finite specification: A dichotomy},
journal = {Adv. Math.},
volume = {405},
pages = {108513},
year = {2022},
issn = {0001-8708},
doi = {https://doi.org/10.1016/j.aim.2022.108513},
url = {https://www.sciencedirect.com/science/article/pii/S0001870822003309},
author = {Frédérique Bassino and Mathilde Bouvel and Valentin Féray and Lucas Gerin and Mickaël Maazoun and Adeline Pierrot}
}

@article{fekete,
title = {{Some linear and some quadratic recursion formulas II}},
journal = {Indag. Math.},
volume = {55},
pages = {152-163},
year = {1952},
issn = {1385-7258},
doi = {https://doi.org/10.1016/S1385-7258(52)50021-0},
author = {N.G. de Bruijn and P. Erdös}
}

@book{woess,
    place={Cambridge}, 
    series={Cambridge Tracts in Math.}, 
    title={Random Walks on Infinite Graphs and Groups}, 
    publisher={Cambridge Univ. Press}, 
    author={Woess, Wolfgang}, 
    year={2000}, 
    collection={Cambridge Tracts in Math.}
}

@article{hyperbolic,
 ISSN = {08940347, 10886834},
 URL = {http://www.jstor.org/stable/43302859},
 author = {Sébastien Gouëzel},
 journal = {J. Amer. Math. Soc.},
 number = {3},
 pages = {893--928},
 publisher = {American Mathematical Society},
 title = {LOCAL LIMIT THEOREM FOR SYMMETRIC RANDOM WALKS IN {G}ROMOV-HYPERBOLIC GROUPS},
 urldate = {2026-02-14},
 volume = {27},
 year = {2014}
}

@article{pattern,
    author = {Kesten, Harry},
    title = {On the Number of Self‐Avoiding Walks},
    journal = {J. Math. Physics},
    volume = {4},
    number = {7},
    pages = {960-969},
    year = {1963},
    month = {07},
    issn = {0022-2488},
    doi = {10.1063/1.1704022},
    }

@misc{pak2022,
      title={Algebraic and arithmetic properties of the cogrowth sequence of nilpotent groups}, 
      author={Igor Pak and David Soukup},
      year={2022},
      eprint={2210.09419},
      archivePrefix={arXiv},
      primaryClass={math.GR},
      url={https://arxiv.org/abs/2210.09419}, 
}

@article{bodart24,
author = {Bodart, Corentin},
title = {A virtually nilpotent group whose Green series is not {D}-finite},
journal = {Internat. J. Algebra Comput.},
volume = {0},
number = {0},
pages = {1-11},
year = {0},
doi = {10.1142/S0218196726500074}
}

\appendix
\section{Explicit asymptotics for some examples}
\label{appx:numerics}
In this appendix, we discuss some explicit asymptotic results for small values of our parameters.

\subsection{$B_3$ with trefoil knot presentation}
\label{appx:torus_knot}

Consider the following presentation of the trefoil group
\begin{align}
  \langle c,d \mid c^3=d^2 \rangle.
\end{align}

Then, using computer algebra systems, we obtain the following.
\begin{theorem}\label{thm F c3d2}
 The generating function $F_2(z;q)$ satisfies
\begin{align}
 0 &= z^2-1 + j_1 \cdot F_2(z;q) + j_2 \cdot F_2(z;q)^2 + j_3 \cdot F_2(z;q)^3,
\end{align}
where
\begin{subequations}
  \begin{align}
    j_1 &= 1-(Q+3)z^2+Qz^3-Qz^4, \\
 j_2 &= -1+(2Q+8)z^2+Qz^3 -(Q^2+4Q+7)z^4+(Q+1)Qz^5, \\
 j_3 &= -1+(3Q+12)z^2+2Qz^3 - (3Q^2+12Q+21)z^4+6(Q+1)Qz^5, \\
 Q &= q + q^{-1}.
  \end{align}
\end{subequations}
\end{theorem}

\begin{cor}\label{cor-torus-cogrowth}
The cogrowth of the group $\langle c,d \mid c^3=d^2 \rangle$ is the largest positive zero of the polynomial
$m^4-2m^3-11m^2+12m+4$. That is
\begin{align}
\lim_{n\longrightarrow\infty}\left(\#\text{trivial words of length $n$}\right)^{1/n}
&= \frac{1+\sqrt{25+16\sqrt{2}}}{2} \approx 3.950630994\dots
\end{align}
\end{cor}
\begin{proof}
 By Lemma~\ref{lem:growth_rate_pop} we know that the coefficients of the cogrowth series and those of $F_2(z;1)$ grow at the same exponential rate. So it suffices to examine the asymptotics of $F_2(z;1)$. Substituting $q=1$ into the equation in
Theorem~\ref{thm F c3d2} gives:
\begin{multline}
 \bigg( (-12z^3+11z^2+2z-1)F_2(z;1)^2+(-2z^3+z^2-z)F_2(z;1)+1-z \bigg) \\
 \times
 \bigg( (1+2z-3z^2)F_2(z;1)+z+1\bigg) = 0.
\end{multline}
Since we know that $F_2(z;1)$ is a power series with non-negative coefficients, we can discard the solution to the linear term and one of the solutions to the quadratic term, leaving
\begin{align}
 F_2(z;1) &= \frac{-2z^3+z^2-z+(2-z)\sqrt{1-2z-11z^2+12z^3+4z^4}}{2(1-z)(1+3z)(1-4z)}.
\end{align}
Thus, one expects singularities at $1,-1/3, 1/4$ and the zeros of $1-2z-11z^2+12z^3+4z^4$. The apparent singularity at $z=1/4$ does not occur since the numerator is also zero at $z=1/4$. A little computer algebra then shows that the dominant singularity comes from the smallest positive zero of $1-2z-11z^2+12z^3+4z^4$.
\end{proof}

We note here that we were able to find a 7th-order linear differential equation satisfied by the cogrowth series, which is equivalent to a 13th-order recurrence satisfied by its coefficients.

We can go further and use the methods of \cite{drmota1997} to compute more detailed asymptotics of the coefficients. In particular, we can establish the asymptotics of $[z^n q^{\alpha n}] F(z;q)$ for a given constant $\alpha$; these are
sufficient to demonstrate that the cogrowth series is D-finite but not algebraic.

\begin{theorem}\label{thm torus gaussian}
 Let $x$ be a real number, then
 \begin{align}
 [z^n q^{ \alpha \sqrt{n}}] F_2(z;q) &= \frac{\gamma}{\pi \sigma n^2} \mu^n e^{-\alpha^2/2\sigma^2} \cdot \left(1 + o(1)\right), 
\end{align}
where $\mu$ is the cogrowth, being the largest positive solution of
\begin{align}
 m^{4}-2m^{3}-11m^{2}+12m+4 &=0 && \text{ and so } && \mu = \frac{1+\sqrt{25+16\sqrt{2}}}{2}.
\end{align}
$\sigma^2$ is the smallest positive solution of
\begin{align}
 452s^4-904s^3+512s^2-60s-1 &=0,
\end{align}
and the constant $\gamma^2$ is the largest positive solution of
\begin{align}
65536g^4 - 41566208 g^3 + 5469075424 g^2 + 914873312 g + 1442897 &= 0.
\end{align}
Numerically these are
\begin{align}
 \mu &\approx 3.950630994\dots &
 \sigma^2 &\approx 0.1801879352\dots &
 \gamma &\approx 21.16215950 \dots.
\end{align}

Hence, the coefficients of the cogrowth series $[z^n q^0]F(z;q)$ scale with an exponent of $-2$. Consequently, the cogrowth series is not algebraic.
\end{theorem}
\begin{proof}
We follow the arguments used in \cite{drmota1997}. First, write the polynomial equation satisfied by $F_2$ as $P(F_2,z;q)=0$. We can use this to find the dominant singularity, $z_c$, of $F_2(z;1)$ and the value of the generating function at that point $F_c = F_2(z_c;1)$. We can do this from the explicit expression for $F_2(z;1)$ (which is given in the proof of Corollary~\ref{cor-torus-cogrowth}) or as follows.

Let $(z,F)=(z_c, F_c)$ be the unique positive solution of the equations
\begin{align}
 P(F,z;1) &= 0, & P_F(F,z;1)&=0,
\end{align}
where we have used $P_F$ to denote the partial derivative of $P$ with respect to $F$. We can readily compute these using
computer algebra, and find that $z_c$ is the smallest positive solution of
\begin{align}
4z^4+12z^3-11z^2-2z+1 &= 0,
\end{align}
while $F_c$ is the largest zero of
\begin{align}
2F^4-4F^3-67F^2+20F+2 &= 0.
\end{align}
Numerically, these are
\begin{align}
 z_c &\approx 0.2531241216\dots & F_c&\approx 6.744148958\dots
\end{align}

Since $P_z(F_c,z_c,1)\neq 0$ and $P_{FF}(F_c,z_c,1) \neq 0$ we can then expand $F_2(z;q)$ around $(z,q)=(z_c,1)$ as
\begin{align}
 F_2(z;q) &= \alpha(q) - \beta(q) \sqrt{1-\frac{z}{\rho(q)} } + \mathcal{O}(z-\rho(q)),
\end{align}
where $\alpha(q), \beta(q)$ are analytic functions around $q=1$ and $\rho(q)$ is the radius of convergence. We refer the reader to \cite{drmota1997} for the details of this argument. Further,
$\alpha(q) = F_2(\rho(q),q)$ and
\begin{align}
 \beta(q) &= \sqrt{ \frac{2\rho(q) P_z\big( F_2(\rho(q);q)),\rho(q);q \big)}{P_{FF}\big(F_2(\rho(q);q)),\rho(q);q\big)} }.
\end{align}
A little computer algebra shows that $\rho(q)$ satisfies a degree 8 polynomial equation (which can be found by computing the discriminant of $P(F_2;z,q)$ with respect to $F_2$).

Singularity analysis (see \cite{Flajolet} Theorem IX.12) then tells us that
\begin{align}
 [z^n] F_2(z;q) &= \frac{\beta(q)}{2\sqrt{\pi}} \cdot \rho(q)^{-n} n^{-3/2} (1 + o(1)).
\end{align}
Saddle point methods then allow us to extract the coefficients of $q$:
\begin{align}
 [z^n q^{\alpha \sqrt{n}}]F_2(z;q) &= \frac{\beta(1)}{\sigma \pi \sqrt{8} } \cdot \rho(1)^{-n}
\cdot n^{-2} e^{-\alpha^2/2\sigma^2} (1 + o(1)),
\end{align}
where
\begin{align}
 \sigma^2 &= \left( \frac{\rho'(1)}{\rho(1)} \right)^2 - \frac{\rho''(1)}{\rho(1)} - \frac{\rho'(1)}{\rho(1)}.
\end{align}
This can be computed from the polynomial satisfied by $\rho(q)$, and shows that $\sigma^2$ is the largest positive solution of
\begin{align}
 452s^4-904s^3+512s^2-60s-1 &=0.
\end{align}
Numerically, this gives $\sigma \approx 0.1801879352\dots$.

The constant $\gamma = \beta(q)/\sqrt{8}$ can be computed from $z_c$ and $F_c$. The largest positive solution of
\begin{align}
65536g^4 - 41566208 g^3 + 5469075424 g^2 + 914873312 g + 1442897 &= 0
\end{align}
gives $\gamma^2$, and numerically $\gamma \approx 21.16215950\dots$.

This scaling form indicates that the coefficients of the cogrowth series scale with a correction of $n^{-2}$, which is incompatible with the series being algebraic.
\end{proof}

\subsection{$B_3$ with standard presentation}
\label{appx:braid}

Consider the following presentation of the 3-strand braid group $B_3$:
\begin{align}
  \langle a, b \mid aba = bab \rangle.
\end{align}

Then, one has, using computer algebra systems:
\begin{theorem}\label{thm f artin}
 The generating function $F(z;q)$ satisfies
\begin{align}
(8(q+q^{-1})z^3 + 12z^2 - 1)F^3+(2z-1)(2z+1)F^2+F+1 &= 0.
\end{align}
\end{theorem}

\begin{cor}
 The cogrowth of $B_3$ with respect to the generating set $\{a,a^{-1},b,b^{-1}\}$ is the largest positive zero of the polynomial
$m^2-2m-7$. That is
\begin{align}
\lim_{n\longrightarrow\infty}\left(\#\text{trivial words of length $n$}\right)^{1/n}
&= 1+2\sqrt{2} \approx 3.828427124\dots
\end{align}
\end{cor}
\begin{proof}
 We again appeal to Lemma~\ref{lem:growth_rate_pop}, which also holds for this group, and the proof is identical. So it is sufficient to study $F(z;1)$. Setting $q=1$ into the equation for $F$ in Theorem~\ref{thm f artin} gives
\begin{align}
(2Fz+F+1)(8F^2z^2+2F^2z-F^2-2Fz+1) &=0.
\end{align}
Of the 3 possible solutions, only one is a positive generating function:
\begin{align}
 F(z,1) &= \frac{-z+\sqrt{1-2z-7z^2}}{(1-4z)(1+2z)}.
\end{align}
Standard singularity analysis then shows that the exponential growth rate is determined by the reciprocal of the smallest positive zero of $1-2z-7z^2$; the apparent singularity at $z=1/4$ is not present.
\end{proof}

We note that we were able to find a 6th-order linear differential equation satisfied by the cogrowth series, which is equivalent to a 9th-order recurrence satisfied by its (even) coefficients. 

Again, as done earlier, we can also obtain finer details of the asymptotics of the coefficients of $F(z;q)$. This is somewhat complicated by the fact that the coefficients of $F$ display strong odd-even behaviour; in particular, $[z^n q^m]F(z;q)$ is only non-zero when $n+m$ is even.
\begin{theorem}
 Let $\alpha$ be a real number, so that $m=\alpha\sqrt{n}$ is an integer, then
 \begin{align}
 [z^n q^{\alpha \sqrt{n}}] F(z;q) &= \frac{\gamma}{\pi \sigma n^2} \cdot (1+(-1)^{n+m}) \cdot \mu^n \cdot
e^{-\alpha^2/2\sigma^2} \cdot \left(1 + o(1)\right),
\end{align}
where
\begin{align}
 \mu &= 1+2\sqrt{2}, &
 \sigma &= \sqrt{\frac{5-3\sqrt{2}}{7}}, &
 \gamma &= (6+5\sqrt{2})\sqrt{10-\sqrt{2}}.
\end{align}
Numerically these are
\begin{align}
 \mu &\approx 3.828427127\dots &
 \sigma &\approx 0.3289288492\dots &
 \gamma &\approx 38.30020588 \dots.
\end{align}

Hence, the coefficients of the cogrowth series $[z^n q^0]F(z;q)$ scale with an exponent of $-2$ and consequently the cogrowth series is not algebraic.
\end{theorem}
The odd-even behaviour of $F(z;q)$ (as reflected in the statement of the theorem) makes the proof a little more complicated than that of Theorem~\ref{thm torus gaussian}. In particular, we must take into account the contributions of two saddle points.
\begin{proof}
We follow the proof of Theorem~\ref{thm torus gaussian}. Following \cite{drmota1997}, let $P(F,z,q)$ be the polynomial equation satisfied by $F(z;q)$.Then its discriminant is
\begin{align}
D_F(z,q) &= \left( 27{q}^{4}+58{q}^{2}+27 \right) {z}^{4}+72q \left( {q}^{2}+1 \right) {z}^{3}+44{z}^{2}{q}^{2}-4q
\left( {q}^{2}+1 \right) z -4{q}^{2}.
\end{align}
Examining the derivative of $z$ with respect to $q$, we see that there are two relevant saddles, which gives the main complication in this proof. When $q=1$ we have an exponential growth of $\mu^n$, while at $q=-1$ the growth is $(-\mu)^n$. The combination of these terms gives the required odd-even behaviour.

From the saddles at $q=\pm 1$ we obtain (recalling that $\alpha\sqrt{n}=m \in \Z$)
\begin{align}
 [z^n q^{\alpha\sqrt{n}}] &= \frac{\beta(1)}{\sigma \pi \sqrt{8} } \cdot \mu^n \cdot n^{-2} e^{-\alpha^2/2\sigma^2} (1
+ o(1)),  \\
 [z^n q^{\alpha\sqrt{n}}] &= \frac{\beta(\pm 1)}{\sigma \pi \sqrt{8} } \cdot (-1)^{n+m} \mu^n \cdot n^{-2}
e^{-\alpha^2/2\sigma^2} (1 + o(1)), 
\end{align}
where
\begin{align}
 \beta(\pm 1) = 4 \sqrt{44+25\sqrt{2}} \quad \text{ and } \quad \sigma^2 = \frac{5-3\sqrt{2}}{7}.
\end{align}
Adding these contributions gives the stated result.
\end{proof}

\section{$B_3$ with presentation \(\langle a, x \mid axa = x^2 \rangle\)}
\label{appx:axa_group}
Let $\Delta = x^3$, then the Schreier graph of the group with respect to the 
 presentation \(\langle a, x \mid axa = x^2 \rangle\)
has the structure shown in Figure~\ref{fig:axa_overall}.
 \begin{figure}[H]
 \begin{center}
 \includegraphics[width = 0.5\textwidth]{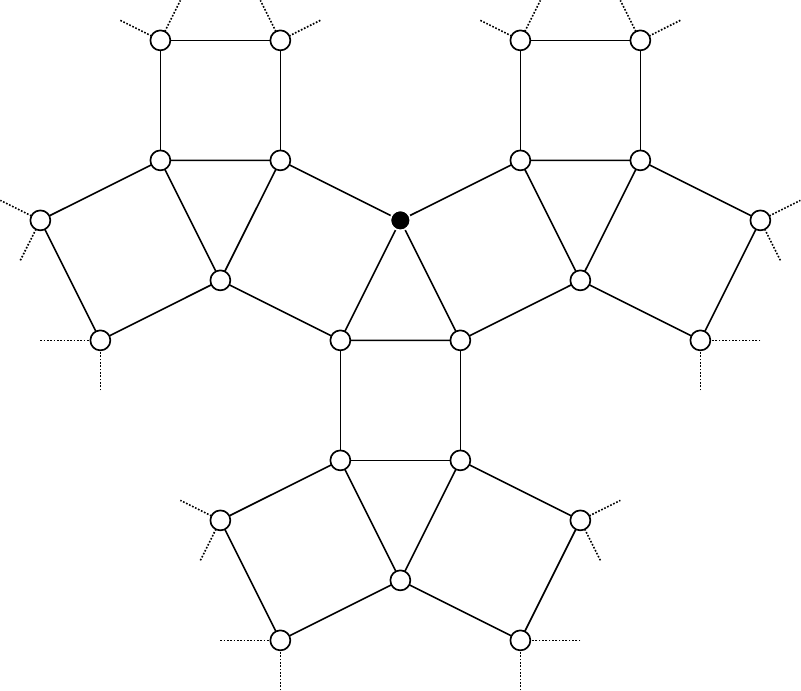}
 \end{center}
 \caption{The Schreier graph of the group \(\langle a,x \mid axa=x^2 \rangle\); The coset $\langle \Delta \rangle$ is indicated by the black vertex.}
 \label{fig:axa_overall}
 \end{figure}

\subsection{Counting returns on the Schreier graph (no $q$'s)}
We start by first counting returns on the Schreier graph without considering the exponent of $\Delta$. Consider the central triangle shown in Figure~\ref{fig:axa_central_triangle}.
\begin{figure}[h]
\begin{center}
 \includegraphics[height=2.5cm]{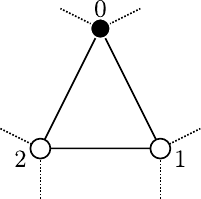}
\end{center}
 \caption{The central triangle of the Schreier graph.}
 \label{fig:axa_central_triangle}
\end{figure}
The vertex corresponding to the coset $\langle \Delta \rangle$ is indicated by the black vertex. We label this vertex 
$0$ and then proceeding clockwise around the face, label the other vertices $1$ and $2$. Now let
\begin{itemize}
 \item $F_{00}(z)$ be the generating function of paths in the graph that start and end at the vertex $0$, 
 \item $F_{01}(z)$ be the generating function of paths in the graph that start at $0$ and end at $1$, and
 \item $F_{02}(z)$ be the generating function of paths in the graph that start at $0$ and end at $2$.
\end{itemize}
In order to establish equations satisfied by these generating functions, we consider similar paths, but confined to two of the three main ``branches'' of the graph. In particular, consider the graph depicted in Figure~\ref{fig twobranches}, which shows the central triangular face and two of the three ``branches'' attached to it.
\begin{figure}[h!]
\begin{center}
 \includegraphics[width = 0.5\textwidth]{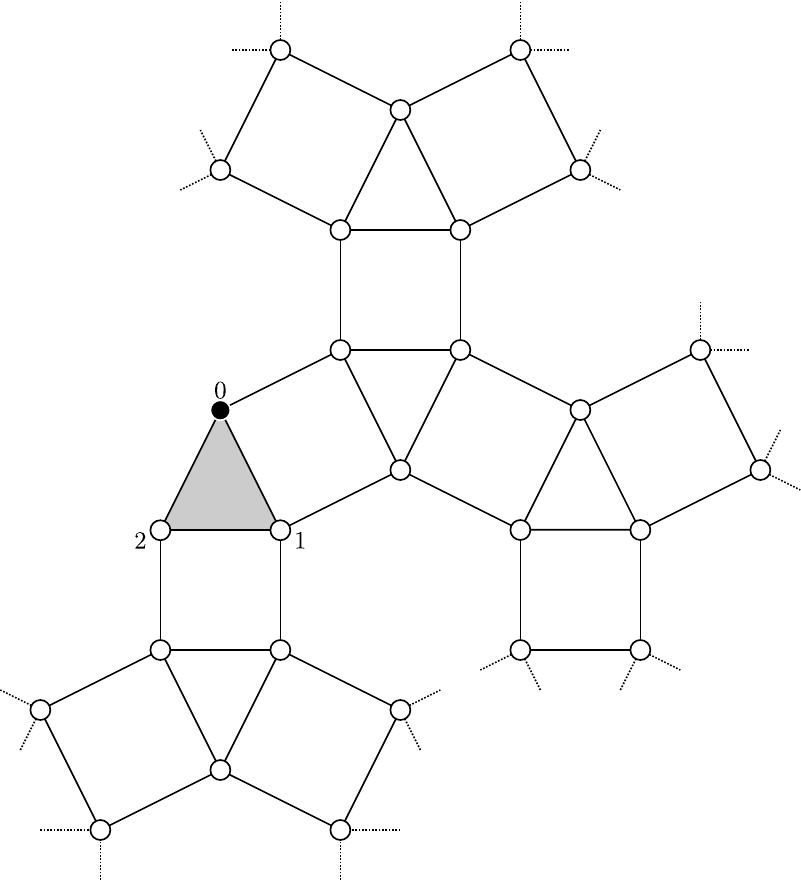}
\end{center}
\caption{The graph with one ``branch'' removed.}
\label{fig twobranches}
\end{figure}

Let
\begin{itemize}
 \item $G_{00}(z)$ be the generating function of paths in this graph that start and end at the vertex $0$, 
 \item $G_{01}(z)$ be the generating function of paths in this graph that start at $0$ and end at $1$, and
 \item $G_{02}(z)$ be the generating function of paths in this graph that start at $0$ and end at $2$.
\end{itemize}
The generating functions $G_{10}$ and $G_{20}$ are defined analogously.

We also need to define some ``primitive loops''. Let 
\begin{itemize}
 \item $L_{00}(z)$ be the generating function of walks of positive length that start at $0$, never visit $1$ or $2$ 
and only return to $0$ at their last step. Similarly, define
\item $L_{01}(z)$ be the generating function of walks that start at $0$, never return to $0$, never visit $2$ and 
only return to $1$ at their last step, and
\item $L_{10}(z)$ be the generating function of walks that start at $1$, never return to $1$, never visit $2$ and 
only return to $0$ at their last step.
\end{itemize}
\begin{lemma}
 The $G_{ij}$ generating functions satisfy
 \begin{align}
 G_{00}(z) &= 1 + G_{00}(z) L_{00}(z) + G_{01}(z) L_{10}(z) + z G_{02}(z), \\
 G_{01}(z) &= G_{00}(z) L_{01}(z) + G_{01}(z) \cdot 2 L_{00}(z) + G_{02}(z)L_{10}(z), \\
 G_{02}(z) &= zG_{00}(z) + G_{01}(z)L_{01}(z) + G_{02}(z)L_{00}(z), 
\end{align}
and, further,
\begin{align}
 G_{10}(z) &= 2L_{00}(z) G_{10}(z) + L_{10}(z) G_{00}(z) + L_{01}(z) G_{20}(z), \\
 G_{20}(z) &= L_{00}(z) G_{20}(z) + L_{10}(z) G_{10}(z) + z G_{00}(z).
\end{align}
\end{lemma}
\begin{proof}
The first three equations are obtained by decomposing the corresponding paths at the last time they leave any of the vertices $0,1$ or $2$. We consider each equation in order.
\begin{itemize}
 \item Take any path counted by $G_{00}$. If it has length $0$, then it contributes $1$ to the generating function. Otherwise, we can decompose the path by the last time it leaves any of the vertices $0$, $1$ and $2$. 
\begin{itemize}
 \item If it leaves $0$ last, then cutting the path at that point breaks it into a path counted by $G_{00}$ and another counted by $L_{00}$. Thus it contributes $G_{00} L_{00}$.
\item If it leaves $1$ last, then it cannot step to the vertex $2$. So, so cutting the path at that point breaks the path into one counted by $G_{01}$ and a primitive loop counted by $L_{10}$. Hence it contributes $G_{01} L_{10}$.
\item If it leaves $2$ last, then it cannot step to the vertex $1$ and so must simply take one step back to the vertex 
$1$. Hence it contributes $zG_{02}$.
\end{itemize}
\item Consider a path counted by $G_{01}$ and cut it at the last time it leaves $0,1$ or $2$.
\begin{itemize}
 \item If it leaves $0$ last then the path separates into one counted by $G_{00}$ and a primitive loop counted by $L_{01}$.
\item If it leaves $1$ last then the path separates into one couned by $G_{01}$ and a loop counted by $L_{00}$. However, this loop can either reside in the ``branch'' attached to the vertices $0,1$ or the vertices $1,2$. Hence this contributes $G_{01} \cdot 2L_{00}$.
\item If it leaves $2$ last then the path decomposes into one counted by $G_{02}$ and a loop counted by $L_{10}$, giving the $G_{02}L_{10}$ term.
\end{itemize}

\item Take any path counted by $G_{02}$ and decompose it by the last time it leaves $0, 1$ or $2$. 
\begin{itemize}
\item If it leaves $0$ last, then, since it cannot visit $1$, it must simply take a single step to $2$. Thus, it contributes $zG_{00}$.
\item If it leaves $1$ last then decomposing it at the last time it leaves $1$ we obtain a path counted by $G_{01}$ and one counted by $L_{01}$, and so it contributes $G_{01}L_{01}$.
 \item If it leaves $2$ last, then cutting the path at that point breaks it into a path counted by $G_{02}$ and (by symmetry) another counted by $L_{00}$. Thus, it contributes $G_{02}L_{00}$.
\end{itemize}
\end{itemize}
This gives us the relations that we require.
\end{proof}
We can now express the $L$ generating functions in terms of the $G_{ij}$ generating functions.
\begin{lemma}
\label{lem:axa_L_G_decomp}
The $L$ generating functions satisfy 
\begin{align}
 L_{00}(z) &= z^2 G_{00}(z), &
 L_{01}(z) &= z + z^2 G_{02}(z), &
 L_{10}(z) &= z + z^2 G_{20}(z).
\end{align}
\end{lemma}
The last two equations are obtained by considering the \emph{first} time the corresponding paths return to any of $0,1$, or $2$. The arguments are then very similar.
\begin{proof}
Consider the graph in Figure~\ref{fig:axa_decomp}. We consider each equation in turn.

\begin{figure}
\begin{center}
 \includegraphics[width=0.45\textwidth]{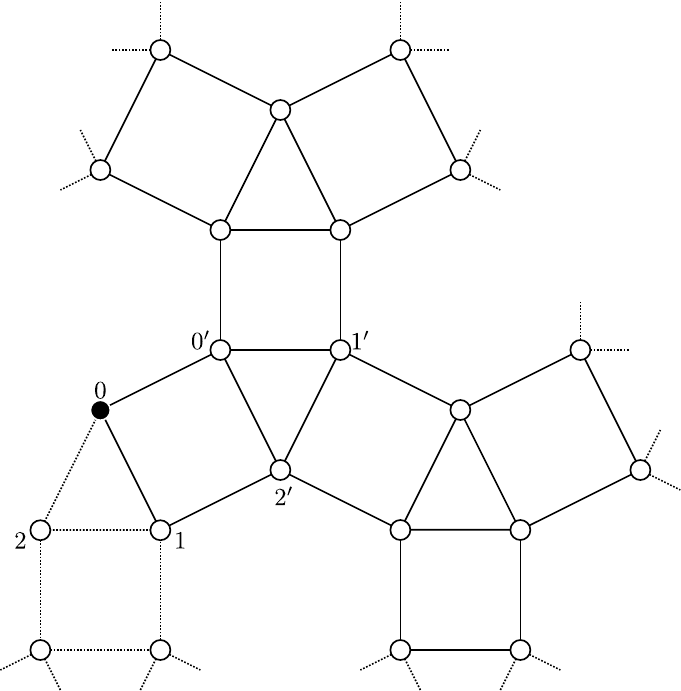}
\end{center}
\caption{Graph with labelled vertices used to prove the recursive decomposition of paths for the proof of Lemma~\ref{lem:axa_L_G_decomp}}.
\label{fig:axa_decomp}
\end{figure}

\begin{itemize}
 \item Any path counted by $L_{00}$ must take at least 2 steps. The first step must go from $0$ to $0'$ and the last from $0'$ to $0$. Deleting these steps, one has a walk that starts and ends at $0'$ and never visits the vertex $1$. These are precisely the walks counted by the generating function $G_{00}$. This gives \(L_{00}(z) = z^2 G_{00}(z)\) as required.
\item Any path counted by $L_{01}$ either has length $1$ (contributing $z$) or takes at least 3 steps. The first of these steps must be from $0$ to $0'$, while the last is from $2'$ to $1$. Deleting these two steps gives a walk that is counted by $G_{02}$. Hence \(L_{01}(z) = z + z^2 G_{02}(z)\).
\end{itemize}
A similar argument gives the equation for $L_{10}$.
\end{proof}

We can now express the $F$ generating functions in terms of the $L$ generating functions.
\begin{lemma}
 The $F$ generating functions satisfy
 \begin{align}
 F_{00}(z) &= 1 + F_{00}(z) \cdot 2 L_{00}(z) + F_{01}(z)L_{10}(z) + F_{02}(z) L_{01}(z), \\
 F_{01}(z) &= F_{00}(z) L_{01}(z) + F_{01}(z) \cdot 2 L_{00}(z) + F_{02}(z) L_{10}(z), \\
 F_{02}(z) &= F_{00}(z) L_{10}(z) + F_{01}(z) L_{01}(z) + F_{02}(z) \cdot 2 L_{00}(z).
\end{align}
\end{lemma}
\begin{proof}
As was the case for the $G_{ij}$-equations, all three equations are obtained by decomposing the corresponding paths at the last time they leave any of the vertices $0,1$, or $2$. Let us start with the first equation.

Consider a path counted by $F_{00}$. It either has length $0$ or we cut it at the last time it leaves any of the vertices $0,1$ or $2$.
\begin{itemize}
 \item If it leaves $0$ last then it separates into a path counted by $F_{00}$ and a primitive loop. Since this loop cannot visit either $1$ or $2$, it must be counted by $L_{00}$. However, the loop can either reside in the ``branch'' attached to $0$ and $1$ or the ``branch'' attached to $0$ and $2$. Hence, this contributes $F_{00} \cdot 2 L_{00}$.
\item If the path leaves $1$ last then it separates into a path counted by $F_{01}$, and a primitive loop starting at $1$, ending at $0$, and never visiting any of $0,1,2$ in between. This is precisely a path counted by $L_{10}$. Hence, these paths contributes $F_{01} L_{10}$.
\item If the path leaves $2$ last, then it decomposes into a path counted by $F_{02}$ and a loop. The loop starts at $2$ and ends at $0$ never visiting $0,1,2$ in between. These loops are (by symmetry) counted by the generating function $L_{01}$. Hence, these paths give the term $F_{02} L_{01}$.
\end{itemize}
The equations for $F_{01}$ and $F_{02}$ follow by very similar arguments.
\end{proof}

These three lemmas define a system of 11 algebraic equations. However, we wish to focus on $F_{00}(z)$. We can use computer algebra to eliminate the other generating functions to obtain a single polynomial equation satisfied by $F_{00}(z)$:

\begin{theorem}
 The generating function $F_{00}(z)$ satisfies the following quintic equation
 \begin{align}
 0 =\; & ( 2z+1 ) ( z-1 ) ( 3z+1 ) ( -1+4z )^{2} ( 4z+1 ) F_{00}(z)^{5} \notag\\
& + ( -1+4z ) ( 16{z}^{5}-28{z}^{4}-50{z}^{3}-5{z}^{2}+6z+1 ) F_{00}(z)^4 \notag\\
&+ ( 52{z}^{4}+20{z}^{3}-18{z}^{2}-2z+1 ) F_{00}(z)^3 \notag\\
&+ ( 12{z}^{4}-12{z}^{3}-18{z}^{2}+2z+1 ) F_{00}(z)^2 \notag\\
&+z ( z-2 ) ( 4{z}^{2}+2z-1 ) F_{00}(z) \notag\\
&-{z}^{2} ( 2 z+3 ). \label{eqn axa_q1}
\end{align}
Hence, the asymptotics of the number of paths of length $n$ that start and end at the origin are given by
\begin{align}
 [z^n] F_{00}(z) &= A \cdot \mu^n n^{-3/2} \left(1 + o(1) \right).
\end{align}
where $\mu \approx 3.9076667\ldots$ is the largest positive real solution of
\begin{align}
  0 = \; &4z^{11} + 24z^{10} - 88z^9-763z^8-130z^7+6598z^6+9136z^5 \notag\\
 &-8940z^4-12888z^3+7788z^2+1192z-108. 
\end{align}
\end{theorem}
\begin{proof}
By symmetry:
\begin{align}
 G_{10} &= G_{01}, & G_{20} &=G_{02}, &
 L_{10} &= L_{01}, & F_{02} &= F_{01}.
\end{align}
So our original system is reduced to a system of 7 equations in 7 unknown functions. One can then use computer algebra to eliminate $G_{00}, G_{01}, G_{02}, L_{00}, L_{01}$ and $F_{01}$, leaving the above single equation in $F_{00}(z)$. Standard methods of analytic combinatorics then yield the asymptotic form given.
\end{proof}

\subsection{Return to the Schreier graph}
We proceed in much the same manner as for the groups above. We now consider the graph as a Schreier graph and must keep track of the exponent of $\Delta$ in the normal forms of the associated group elements. Notice that since \(\Delta\) commutes with both generators, we are free to move the power of \(\Delta\) to the extreme right of any normal form. We do that here.

Consider the central triangle again as shown in Figure~\ref{fig:axa_central_cosets}.

\begin{figure}[h]
\begin{center}
 \includegraphics[height=2.5cm]{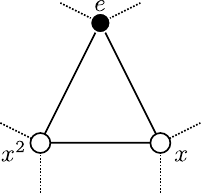}
\end{center}
\caption{The central triangle of the Schreier graph with the cosets $\langle \Delta \rangle$, $x\langle \Delta \rangle$, and \(x^2 \langle\Delta\rangle\).}
\label{fig:axa_central_cosets}
\end{figure}

The vertices labelled $e,x,x^2$ correspond to the cosets $\langle \Delta \rangle$, $x\langle \Delta \rangle$, and $x^2\langle \Delta \rangle$. Now we extend the generating functions used above to include information about $\Delta$. As before, $z$ counts the length and $q$ counts the exponent of $\Delta$ in the corresponding normal form. Let
\begin{itemize}
 \item $F_{00}(z,q)$ be the generating function of paths that start at the identity and end at $\langle \Delta \rangle$, 
 \item $F_{01}(z,q)$ be the generating function of paths that start at the identity and end at $x \langle \Delta \rangle$, and
 \item $F_{02}(z,q)$ be the generating function of paths that start at the identity and end at $x^2 \langle \Delta \rangle$. 
\end{itemize}
We define the extended $G_{ij}$ and $L$ generating functions in a similar manner. A little care is needed for $L_{10}(z,q)$, $G_{10}(z,q)$, and $G_{20}(z,q)$. Let
\begin{itemize}
 \item $L_{10}(z)$ be the generating function of walks of positive length that start at $x$, never visit $1$ or $2$, and only return to the coset $\langle \Delta \rangle$ at their last step. Then let \item $G_{10}(z)$ be the generating function of walks on the Schreier graph, corresponding to the graph in Figure~\ref{fig twobranches}, that start at $x$ and end at the coset $\langle \Delta \rangle$. Finally, let \item $G_{20}(z)$ be the generating function of walks on the Schreier graph, corresponding to the graph in Figure~\ref{fig twobranches}, that start at $x^2$ and end at the coset $\langle \Delta \rangle$. 
\end{itemize}

The extended generating functions satisfy very similar equations, which are obtained by very similar arguments.
\begin{lemma}
 The extended $G_{ij}$ generating functions satisfy
 \begin{align}
 G_{00}(z,q) &= 1 + G_{00}(z,q) L_{00}(z,q) + G_{01}(z,q) L_{10}(z,q) + z q G_{02}(z,q), \\
 G_{01}(z,q) &= G_{00}(z,q) L_{01}(z,q) + G_{01}(z,q) \cdot 2 L_{00}(z,q) + G_{02}(z,q)L_{10}(z,q), \\
 G_{02}(z,q) &= zq^{-1} G_{00}(z,q) + G_{01}(z,q)L_{01}(z,q) + G_{02}(z,q)L_{00}(z,q),
 \end{align}
 and further,
\begin{align}
 G_{10}(z,q) &= 2L_{00}(z,q) G_{10}(z,q) + L_{10}(z,q) G_{00}(z,q) + L_{01}(z,q) G_{20}(z,q), \\
 G_{20}(z,q) &= L_{00}(z,q) G_{20}(z,q) + L_{10}(z,q) G_{10}(z,q) + z q G_{00}(z,q).
\end{align}
\end{lemma}
\begin{proof}
We use the same constructions as previously. The equations for $G_{01}(z,q)$ and $G_{10}(z,q)$ follow by very similar arguments. Those for the other three $G$'s also follow very similarly. However, more care is needed to account for changes in the exponent of $\Delta$.
 
Consider any path counted by $G_{00}$ and the last time that path leaves any of the cosets $\langle \Delta \rangle$, $x\langle \Delta \rangle$, or $x^2 \langle \Delta \rangle$.
\begin{itemize}
 \item If it last leaves from $\langle \Delta \rangle$, then the path decomposes into a path from the identity to $\langle \Delta \rangle$, and then a primitive loop starting and ending at $\langle \Delta \rangle$. This gives $G_{00}L_{00}$.

\item If it last leaves from $x\langle \Delta \rangle$, then the path decomposes into a path from the identity to $x \langle \Delta \rangle$, and then a primitive loop starting at $x\langle \Delta \rangle$ and ending at $\langle \Delta \rangle$. This gives $G_{01}L_{10}$.
 \item If it last leaves from $x^2 \langle \Delta \rangle$, then the path decomposes into a path from the identity to $x^2 \langle \Delta \rangle$, and then it must take a single step from $x^2 \langle \Delta \rangle$ to $\langle \Delta \rangle$. This single step must be an $x$ and hence increases the exponent of $\Delta$ by $1$.
\end{itemize}

Now consider any path counted by $G_{02}$ and the last time that path leaves any of the cosets $\langle \Delta \rangle$, $x\langle \Delta \rangle$, or $x^2 \langle \Delta \rangle$.
\begin{itemize}
\item If it last leaves from $\langle \Delta \rangle$, then the path decomposes into a path from $x^2$ to $\langle \Delta \rangle$, and then it must take a single step from $\langle \Delta \rangle$ to $x^2 \langle \Delta \rangle$. This single step must be an $x^{-1}$ and hence decreases the exponent of $\Delta$ by $1$, since $\Delta^k x^{-1} = x^2 \Delta^{k-1}$.

\item If it last leaves from $x\langle \Delta \rangle$ then the path decomposes into a path from the identity to $x \langle \Delta \rangle$, and then a primitive loop starting at $x\langle \Delta \rangle$ and ending at $x^2 \langle \Delta \rangle$. This gives $G_{01}L_{01}$.

\item If it last leaves from $x^2 \langle \Delta \rangle$, then the path decomposes into a path from the identity to $\langle \Delta \rangle$, and then a primitive loop starting and ending at $x^2 \langle \Delta \rangle$. This gives $G_{02}L_{00}$.
\end{itemize}

Finally, take any path counted by $G_{20}$ and examine the first time that path returns to any of the cosets $\langle \Delta \rangle$, $x\langle \Delta \rangle$, or $x^2 \langle \Delta \rangle$.
\begin{itemize}
\item If it first returns to $x^2 \langle \Delta \rangle$, then the path decomposes into a primitive loop from $x^2$ to $x^2 \langle \Delta \rangle$, and a path that starts at $x^2 \langle \Delta \rangle$ and ends at $\langle \Delta \rangle$. Hence, this contributes $L_{00} G_{20}$.

\item If it first returns to $x \langle \Delta \rangle$, then it decomposes into a primitive loop from $x^2$ to $x \langle \Delta \rangle$, and a path from $x \langle \Delta \rangle$ to $\langle \Delta \rangle$. The first of these is counted by $L_{10}$, while the second is counted by $G_{10}$; hence, they contribute $L_{10}G_{10}$.

\item Finally, if it first returns to $\langle \Delta \rangle$, then it must consist of a single step from $x^2$ to $\langle \Delta \rangle$. This step must be an $x$, and so the exponent of $\Delta$ increases by $1$. This gives the term $zq G_{00}$.
\end{itemize}
These give us the required generating function relations.
\end{proof}

Similar arguments allow us to find relations and symmetries within all of these generating functions, which we summarise in the following lemma. We state it without proof.
\begin{lemma}
 The extended $L$ generating functions satisfy
\begin{align}
 L_{00}(z,q) &= z^2 G_{00}(z,q), &
 L_{01}(z,q) &= z + z^2 G_{02}(z,q), &
 L_{10}(z,q) &= z + z^2 G_{20}(z,q).
\end{align}
 while the extended $F$ generating functions satisfy
 \begin{align}
 F_{00}(z,q) &= 1 + F_{00}(z,q) \cdot 2 L_{00}(z,q) + F_{01}(z,q)L_{10}(z,q) + q F_{02}(z,q) L_{01}(z,q), \\
 F_{01}(z,q) &= F_{00}(z,q) L_{01}(z,q) + F_{01}(z,q) \cdot 2 L_{00}(z,q) + F_{02}(z,q) L_{10}(z,q), \\
 F_{02}(z,q) &= q^{-1} F_{00}(z,q) L_{10}(z,q) + F_{01}(z,q) L_{01}(z,q) + F_{02}(z,q) \cdot 2 L_{00}(z,q).
\end{align}
Further, the following symmetries hold
 \begin{align}
 G_{10}(z,q) = G_{01}(z,q^{-1}), &\quad \quad G_{20}(z,q) = G_{02}(z,q^{-1}), \\
 L_{10}(z,q) = L_{01}(z,q^{-1}), &\quad \quad F_{02}(z,q) = q F_{01}(z,q^{-1}).
\end{align}
\end{lemma}

Unfortunately, we have not been able to reduce the above system of equations down to a single equation for \(F_{00}(z,q)\). We are able to find degree~\(11\) polynomial equations for \(L_{00}(z;q), L_{01}(z,q)\), and \(L_{10}(z,q)\) and can express \(F_{00}(z;q)\) as a simple rational function of these \(L\)'s. The computer algebra systems we tried were unable to reduce those equations to a polynomial for \(F_{00}(z,q)\); the computation did not terminate after any reasonable time. 

Note that by fixing the value of \(q\) at small integers \(n=-1,1,2,3,\dots\) we were able to guess algebraic equations for \(F_{00}(z,n)\). When \(n=1\), we recover Equation~\eqref{eqn axa_q1} and when \(n=-1\), we find a cubic(!). For other integer values of \(n\) we find polynomials of degree~\(11\) whose coefficients are polynomials in \(z\) of degree up to \(28\). In principle, it would be possible to determine the \(q\)-dependence of these coefficients by guessing at many values of \(n\), but we have not done this.

\begin{theorem}\label{thm:cogrowth-details-axa}
 The generating function $F_{00}(z,q)$ satisfies an algebraic equation. Hence, the cogrowth series for $B_3$ with respect to the generating set $\{a,x\}$, being $[q^0]F_{00}(z,q)$, is a D-finite generating function.
\label{thm:axa_dfinite}
\end{theorem}
\begin{proof}
  Though we have been unable to find an explicit algebraic equation for \(F_{00}(z,q)\), we know that one must exist from the system of equations above. Hence, \(F_{00}(z,q)\) is also D-finite and so its constant term with respect to \(q\) is also D-finite \cite{MR929767}. 
\end{proof}
\end{document}